\documentclass[10pt]{article}

\usepackage{amssymb,amsmath}
\usepackage{amsthm}
\usepackage{esint}
\usepackage{xcolor}
\usepackage{mathrsfs}

\newcommand{\eps}{\varepsilon}

\newcommand{\pa}{\partial}
\newcommand{\na}{\nabla}

\newcommand{\R}{\mathbb{R}}
\newcommand{\PV}{\textnormal{P.V.}\,}

\newcommand{\dist}{{\rm dist}\, }

\def\BV{\mathrm{BV}}

\newcommand{\J}{\mathscr{J}}
\newcommand{\G}{\mathscr{G}}
\newcommand{\F}{\mathscr{F}}
\newcommand{\N}{{\mathcal{N}}}

\def\C{\mathcal C}
\def\H{\mathcal H}
\def\div{{\mathrm {div}}}
\def\vphi{\varphi}

\def\P{\phi_E^\eps}
\def\Pe{ \phi^\eps_{E_\eps}}

\theoremstyle{plain}

\numberwithin{equation}{section}
\newtheorem{theorem}{Theorem}[section]
\newtheorem{lemma}[theorem]{Lemma}
\newtheorem{definition}[theorem]{Definition}
\newtheorem{remark}[theorem]{Remark}
\newtheorem{proposition}[theorem]{Proposition}

\title{$\Gamma$-convergence of some nonlocal perimeters in bounded subsets of  $\R^n$ with general boundary conditions}

\author{Antoine Mellet\thanks{mellet@umd.edu. Partially supported by NSF Grant DMS-2009236.}  and Yijing Wu\thanks{yijingwu@umd.edu}}

\date{University of Maryland \\ Department of Mathematics \\ College Park, MD 20742  USA}

\begin{document}


\maketitle

\begin{abstract}
We establish the $\Gamma$-convergence of some energy functionals describing 
nonlocal attractive interactions in bounded domains. 
The interaction potential solves an elliptic equation (local or nonlocal) in the bounded domain and 
the primary interest of our results is to identify the effects that the boundary conditions imposed on the potential
 have on the limiting functional. 
We consider general Robin boundary conditions, which include Dirichlet and Neumann conditions as particular cases.
Depending on the order of the elliptic operator the limiting functional involves the usual perimeter or some fractional perimeter.

We also consider the $\Gamma$-convergence of a related energy functional combining the usual perimeter functional and the nonlocal repulsive interaction energy.

\end{abstract}
\medskip

\noindent{\bf Keywords:} Perimeters, Non-local perimeters, $\Gamma$-convergence, Interaction energy.
\medskip

28A75, 35J20, 49Q20

\section{Introduction}
Given a subset $\Omega$ of $\R^n$ (with possibly $\Omega=\R^n$), we consider  the 
nonlocal attractive interaction energy defined for a set $E\subset \Omega$ by
$$V_{K_\eps}(E)  =  - \int_{\Omega} \int_{\Omega}  K_\eps(x,y) \chi_{E} (x) \chi_E(y)\, dx\, dy $$
where $\chi_E$ denotes the characteristic function of the set $E$ and
for all $y\in\Omega$, the kernel $x\mapsto K_\eps(x,y)$ is the solution of 
\begin{equation}\label{eq:Keps0}
K_{\epsilon}+\eps^s(-\Delta)^{s/2} K_{\eps}=\delta(x-y) \quad \mbox{ in } \Omega 
\end{equation}
with $s\in(0,2]$.
Our main interest is when $\Omega \neq \R^n$ and \eqref{eq:Keps0} is supplemented by boundary conditions - 
we will consider general Robin boundary conditions, which includes both Neumann and Dirichlet conditions as particular cases. 
Indeed, while the asymptotic behavior of such functionals  is a classical problem when $K_\eps(x,y) = K_\eps(x-y)$, it does not seem that the effects of the boundary conditions on $\pa\Omega$ (or in $\mathcal C\Omega$ when $s<2$) have been previously identified.

\medskip

We immediately notice that $K_\eps \to \delta$ when $\eps\to0$ and so $V_{K_\eps}(E) \to - |E|$. Since the applications we have in mind have a volume constraint, we are interested in the behavior or $V_{K_\eps}(E) + |E|$ when $\eps\ll1$. 
When $\Omega=\R^n$, we have $K_\eps(x,y) = \eps^{-n} K\left(\frac{x-y}{\eps}\right)$, and this functional   is related to the nonlocal perimeters:
\begin{align}
V_{K_\eps}(E) + |E|& =  \int_{\R^n}\chi _{E} ( 1- {K_\eps} * \chi_E )\, dx
 =  P_{{K_\eps}}(E)\label{eq:VK}
\end{align}
with
$$ 
P_{K}(E) : = \int_E \int_{\R^n\setminus E} {K}(x-y)\, dy\, dx = \frac 1 2 \int_{\R^n}\int_{\R^n} {K}(x-y) |\chi_E(x)-\chi_E(y)|\, dy\, dx.
$$
An important case of such a functional is the $s$-perimeter, denoted $P_s$, which corresponds to 
 $K(x) = \frac{1}{|x|^{n+s}}$ with $s\in(0,1)$, see \cite{cafroq,cafval,figV,val13,Dipierro20}.


The asymptotic behavior of \eqref{eq:VK} when $\eps\ll1$ 
 is a classical topic of research and depends on the rate of decay of $K(x)$ for large $|x|$.
 In our setting, this decay is given by the following lemma (see for instance \cite{MW}):
 \begin{lemma}\label{lem:K}
The solution of \eqref{eq:Keps0} in $\R^n$ is of the form $K_\eps(x,y) = \eps^{-n} K\left(\frac{x-y}{\eps}\right)$ where
$K(z)= k(|z|)\geq 0$ satisfies, as $r\to\infty$:
\begin{equation}\label{eq:Ks}
k(r)\sim
\begin{cases}
 \frac{c_{n,s}}{r^{n+s}}&  \mbox{ if } s\in(0,2)\\[5pt]
c_{n} \frac{e^{-r}}{r^{\frac{n-1}{2}}} & \mbox{ if } s=2
\end{cases}
\end{equation}
with $c_{n,s} = \frac{2^s \Gamma\left(\frac{n+s}{2} \right)}{\pi^{n/2} |\Gamma\left(-\frac s2 \right)|}$.
\end{lemma}
In particular, when $s\in (1,2]$, we have  $|z|K(z)\in L^1(\R^n)$ and 
a classical result of D\'avila \cite{davila}  implies that
$$ 
 P_{{K_\eps}}(E) \sim \eps \sigma P(E)
$$ 
where $P$ denotes the usual perimeter, defined as the total variation of the characteristic function $\chi_E$:
\begin{align*}
P(E)  = \int_{\R^n} |D\chi_E|.
\end{align*}
We recall this result here for the reader's sake:
\begin{theorem}[\cite{davila,BBM}]\label{thm:per1}
Let $K_\eps = \eps^{-n} K(x/\eps)$ where $K(z)= k(|z|)\geq 0$ satisfies
$\int_0^\infty k(r)  r^n\, dr = k_0<\infty$. Then for every $u\in \BV(\R^n)$ with compact support, we have
$$ \lim_{\eps\to 0} \eps^{-1} \int_{\R^n}{\int_{\R^n} {K_{\eps}(x-y)|u(x)-u(y)|dx}dy}  = \sigma  \int_{\R^n} |Du| $$
with
$ \displaystyle\sigma = k_0 \int_{\partial B_1} |e\cdot x|d\H^{n-1}(x)$ (for any $e\in \partial B_1$).
\end{theorem} 
The $\Gamma$-convergence of the nonlocal perimeter $P_{{K_\eps}}$ to the usual perimeter $P$ is proved in this framework in \cite{BP}.
\medskip

On the other hand, Lemma \ref{lem:K} clearly implies that the condition $\int_0^\infty k(r)  r^n\, dr<\infty$ fails when $s\in(0,1]$.
When $\Omega =\R^n$, we expect that the appropriate scaling when $s\in(0,1)$ will give rise to the $s$-perimeter. In fact, various works on this topic suggest (see in particular \cite{muratovsimon} for the case $s=1$):
 $$ 
 V_{K_\eps}(E) + |E| = P_{{K_\eps}}(E) \sim
 \begin{cases}
  \eps^{s} \sigma_s P_s(E) & \mbox{ for } s\in (0,1) \\[4pt]
 \eps  |\ln \eps|   \sigma_1 P(E) & \mbox{ for } s=1.
\end{cases}
$$ 

Our goal in this paper is to make this asymptotic rigorous in the particular framework we consider here (with $K_\eps$ solution of \eqref{eq:Keps0}) and to identify the role of boundary conditions when $\Omega\neq \R^n$.
In the sequel, we denote
$$ \P (x)= \int_\Omega  K_\eps (x,y) \chi_E(y)\, dy$$
and we consider the functional 
$$V_{K_\eps}(E)+ |E|  = \int_\Omega \chi_E (1-\P)\, dx.$$
The potential $\P$ solves
\begin{equation}\label{eq:0}
 \phi + \eps^s (-\Delta)^{s/2} \phi = \chi_E \qquad \mbox{ in } \Omega
\end{equation}
together with Robin boundary conditions. In the local case $s=2$, these are
$$
\alpha \phi + \beta \eps \nabla \phi\cdot n = 0  \quad \mbox{ on } \partial \Omega
$$
and  in the nonlocal case $s\in(0,2)$ the corresponding nonlocal Robin boundary conditions  take the form:
$$
 \alpha \phi  + \beta \widetilde \N (\phi) = 0 \qquad \mbox{ in } \C\Omega
$$
$$ 
 \widetilde \N (\phi)(x)  = \frac{1}{c_{n,s}\int_\Omega \frac{1}{|x-y|^{n+s}} \, dy} \N(\phi) (x) , \quad 
  \N (\phi)(x) =c_{n,s} \int_{\Omega}\frac{\phi(x) - \phi(y) }{|x-y|^{n+s}}\, dy,  \qquad x\in \C\Omega.
 $$
The coefficients $\alpha$ and $\beta$ will always be assumed to be non-negative functions (defined on $\pa\Omega$ or $\mathcal C\Omega$, as appropriate) satisfying $\alpha(x)+\beta(x)>0$ (for the nonlocal case, we refer to \cite{DRV} for an introduction to such boundary value problems with Neumann boundary conditions).

The appropriate scaling depends on $s$ and the discussion above suggests the following definition:
\begin{equation}\label{eq:defJJ}
\J^s_\eps(E):=
\begin{cases}
\displaystyle  \eps^{-s}\int_\Omega \chi_E (1-\P)\, dx & \mbox{ if } s\in(0,1) \\[5pt]
\displaystyle  \eps^{-1}|\ln\eps|^{-1}\int_\Omega \chi_E (1-\P)\, dx   & \mbox{ if } s=1\\[5pt]
\displaystyle  \eps^{-1}\int_\Omega \chi_E (1-\P)\, dx  & \mbox{ if } s\in(1,2].
  \end{cases}
\end{equation}

The main results of this paper identify the limit of $\J^s_\eps(E)$ for a given set $E$ and study the $\Gamma$-convergence of the functional $\J^s_\eps$. 
When $\Omega=\R^n$ some of these results are classical (though in some cases we present simpler proofs that use strongly the fact that the kernel $K_\eps$ solves \eqref{eq:Keps0}), 
but our focus is the case $\Omega \neq \R^n$ and \eqref{eq:Keps0}  (or \eqref{eq:0}) is supplemented by Robin boundary conditions.
To our knowledge, the effects of these boundary conditions on the limiting functional has not been previously identified in all cases $s\in (0,2]$.

\medskip

The original motivation for this study comes from the recent paper \cite{MW}, in which  we study the following energy functional
$$ 
 P(E)  - \beta V_{K_\eps}(E)
$$
defined for Caccioppoli sets $E\subset \R^n$ with fixed volume $|E|=m$.
We thus also identify the $\Gamma$-limit of this functional when $s\in(0,1)$ and $s=2$ (see Theorems \ref{thm:gammaRn}, \ref{thm:gammaOmega} and \ref{thm:gammaOmega2}).

\paragraph{Outline of the rest of the paper}
In Section \ref{sec:main}, we state  the main results of this paper. 
In Section~\ref{sec:pre}, we recall several notations, definition and properties of the perimeter and s-perimeters which will be useful in the paper. We also derive an alternative formula for $\J_\eps^s$ (see Proposition \ref{prop:alt}).
The rest of the paper is then devoted to the proofs of the main results.

\section{Main results of the paper}\label{sec:main}

\subsection{The case $s\in(0,1)$}
We start with the case when $s\in(0,1)$ and $\Omega=\R^n$, that is
$$\J^s_\eps(E)=   \eps^{-s}\int_{\R^n}  \chi_E (1-\P)\, dx $$
where $\P$ is the unique bounded solution of
\begin{equation}\label{eq:P0R}
 \phi + \eps^s (-\Delta)^{s/2} \phi = \chi_E \qquad \mbox{ in } \R^n
 \end{equation}
 (see \eqref{eq:deflaps} for the definition of the fractional Laplacian $(-\Delta)^{s/2}$).
In this case, the kernel $K_\eps$ solution of \eqref{eq:Keps0} is of the form $K_\eps(x,y) = \eps^{-n} K\left(\frac{x-y}{\eps}\right)$ where
$K(z) \sim \frac{c_{n,s}}{|z|^{n+s}}$ as $|z|\to\infty$. 
In that case, the functional $\J^s_\eps$ converges to $c_{n,s}P_s(E)$ with the $s$-perimeter defined by:
$$
P_s(E) : = \frac 1 2  \int_{\R^n} \int_{\R^n}  \frac{|\chi_E(x) - \chi_E(y)|}{|x-y|^{n+s}} \, dx \, dy .
$$
Indeed, we can prove:
 \begin{theorem}\label{thm:1}
Let $\Omega = \R^n$ and $s\in(0,1)$. 
\item[(i)] For every set $E$ in $L^1(\R^n)$ such that $P_s(E)<\infty$, we have 
$$\lim_{\eps\to 0} \J^s_\eps(E) = c_{n,s}   P_s(E) .
$$
\item[(ii)] For any family $\{ E_\eps\}_{\eps>0}$ that converges to $E$ in $L^1(\R^n)$, we have
 $$\liminf_{\eps\to 0} \J^s_\eps(E_\eps) \geq c_{n,s}   P_s(E).$$
In particular, the sequence of functionals $\J^s_\eps$ $\Gamma$-converges to  $c_{n,s}   P_s$.
\end{theorem}
Next, we assume that $\Omega$ is an open subset of $\R^n$ and we 
consider equation \eqref{eq:P0R} in $\Omega$, supplemented with Robin boundary conditions:
\begin{equation}\label{eq:bc}
\begin{cases}
 \phi + \eps^s (-\Delta)^{s/2} \phi = \chi_E \qquad & \mbox{ in } \Omega\\
 \alpha \phi  + \beta \widetilde \N (\phi) = 0 \qquad & \mbox{ in } \C\Omega
\end{cases}
\end{equation}
with
$$ 
 \widetilde \N (\phi)(x)  = \frac{1}{c_{n,s}\int_\Omega \frac{1}{|x-y|^{n+s}} \, dy} \N(\phi) (x) , \quad 
  \N (\phi)(x) =c_{n,s} \int_{\Omega}\frac{\phi(x) - \phi(y) }{|x-y|^{n+s}}\, dy,  \qquad x\in \C\Omega
 $$
where $\alpha$ and $\beta$ are non-negative functions (defined in $\C\Omega$)  satisfying $\alpha(x)+\beta(x)> 0$ for all $x\in \C\Omega$. 
These conditions include the classical Dirichlet boundary conditions when $\beta=0$ and  Neumann boundary conditions when $\alpha=0$ (these Neumann boundary conditions were first introduced and studied in \cite{DRV}).
Before stating our result in this case, we recall  that  the local contribution of the $s$-perimeter in $\Omega$ is defined by:
$$
P^L_s(E,\Omega) : = \frac 1 2  \int_{\Omega} \int_{\Omega}  \frac{|\chi_E(x) - \chi_E(y)|}{|x-y|^{n+s}} \, dx \, dy .
$$
We then prove:
\begin{theorem}\label{thm:2}
Let $\Omega$ be a bounded subset of $\R^n$ such that $P_s(\Omega)<\infty$ and let $s\in (0,1)$.
Assume $\alpha$, $\beta:\C\Omega\mapsto [0,\infty)$  satisfy $\alpha(x)+\beta(x)> 0$ for all $x\in \C\Omega$. Then:
\item[(i)] For every set $E\subset \Omega$ such that  $P_s^L(E,\Omega)<\infty$, we have 
$$\lim_{\eps\to 0} \J^s_\eps(E)  = \J^s_0 (E):= c_{n,s} \left[  P_s^L(E,\Omega) 
+ \int_{\C\Omega}  \frac{\alpha}{\alpha+\beta} \psi_{E }(x)\, dx   + \int_{\C\Omega}   \frac{\beta}{\alpha+\beta}	 \frac{ \psi_{E\cap\Omega} (x)\psi_{\C E \cap \Omega} (x)}{\psi_\Omega(x)} \, dx  \right]
$$
where for a given set $F$, we denote $\psi_F (x): = \int_{F} \frac{1}{|x-y|^{n+s}}\, dy $ (defined for $ x\in \C F$).
\item[(ii)] For any family $\{ E_\eps\}_{\eps>0}$ that converges to $E$ in $L^1(\Omega)$,
 $$\liminf_{\eps\to 0} \J^s_\eps(E_\eps) \geq \J^s_0 (E)
$$
In particular, the sequence of functional $\J^s_\eps$ $\Gamma$-converges to  $ \J^s_0$.
\end{theorem}
For {\em Dirichlet boundary conditions}  (that is when $\beta=0$), we find (using Lemma \ref{lem:Psp})
\begin{align*}
\J^s_0 (E)
& = c_{n,s} \left[   \int_{\R^n} \chi_{\C E\cap\Omega} \psi_{E\cap\Omega}\, dx 
+ \int_{\R^n} \chi_{\C\Omega}  \psi_{E\cap\Omega}(x)\, dx    \right]\\
& = c_{n,s}   \int_{\R^n} \chi_{\C (E\cap\Omega) } \psi_{E\cap\Omega}\, dx \\
& = c_{n,s} P_s(E)
\end{align*}
 so the limiting functional does not see the set $\Omega$.
On the other hand, with {\em Neumann boundary conditions} ($\alpha=0$) our result gives:
\begin{align*}
\J^s_0 (E) &=
c_{n,s} \left[  P_s^L(E,\Omega) 
+  \int_{\C\Omega}   \frac{ \psi_{E\cap\Omega} (x)\psi_{\C E \cap \Omega} (x)}{\psi_\Omega(x)} \, dx  \right]
\end{align*}
which is, to our knowledge, a new functional. Note that this functional takes the same value for the set $E\cap\Omega$ and the set $\C E\cap\Omega$. This symmetry suggests that minimizers are such that $\pa E$ intersect $\pa\Omega$ orthogonally, a common feature of minimal surfaces with Neumann boundary conditions.

\medskip

\subsection{The case $s=2$.}
Next, we consider the classical case $s=2$ for which we have:
$$
\J_\eps(E)= \eps^{-1}\int_\Omega \chi_E (1-\P)\, dx 
$$
(we drop the index $2$ in the local case for clarity)
where $\P$ solves the  local elliptic equation
\begin{equation}\label{eq:02}
 \phi -\eps^2\Delta \phi = \chi_E \qquad \mbox{ in } \Omega.
\end{equation}
When $\Omega = \R^n$, Theorem \ref{thm:per1} applies (in fact, this theorem applies when $\Omega=\R^n$ and $s\in(1,2]$), so the interesting case is when $\Omega$ is a subset of $\R^n$ and  \eqref{eq:02} is supplemented with the following Robin boundary conditions:
\begin{equation}\label{eq:bc0}
\alpha \phi + \beta \eps \nabla \phi\cdot n = 0  \quad \mbox{ on } \partial \Omega.
\end{equation}

In that case, we first prove:
\begin{proposition}\label{prop:4}
Let $s=2$, $\Omega$ be a bounded open set with $C^{1,\alpha}$ boundary. Assume further that 
$\alpha(x)$, $\beta(x)$ are bounded, Lipschitz non-negative functions such that $\alpha(x)+\beta(x) \geq \sigma>0.$
Given a set $E\subset \Omega$ with finite perimeter $P(E,\Omega)<\infty$, we have 
\begin{equation}\label{eq:lims=2}
\lim_{\eps \to 0 } \J_\eps (E) = \frac 1 2 P(E,\Omega)  + 
 \int_{\pa\Omega } \frac{\alpha(x)}{\alpha(x)+\beta(x)} \chi_E(x) \, d\H^{n-1}(x)  .
\end{equation}
\end{proposition}
This proposition identifies the limit of $\J_\eps (E)$. In the particular case of Neumann boundary conditions ($\alpha=0$), we find simply $ \frac 1 2 P(E,\Omega)$, but for Dirichlet boundary conditions, we get
$\frac 1 2 P(E,\Omega)+\int_{\pa \Omega}{\chi_E(x)d\H^{n-1}(x)}$. 
However, this functional is not lower-semicontinuous (take a sequence of sets $E_n$ such that
$ \int_{\pa \Omega}{\chi_{E_n}(x)d\H^{n-1}(x)}=0$ converges to $E$ such that $ \int_{\pa \Omega}{\chi_E(x)d\H^{n-1}(x)} \neq 0$)
so it cannot be the $\Gamma$-limit of $\J_\eps$.
Note that this phenomenon is not restricted to Dirichlet conditions, but occurs whenever  $\alpha >\beta$.
We thus  prove:
\begin{theorem}\label{thm:4}
Let $s=2$, $\Omega$ be a bounded open set with $C^{1,\alpha}$ boundary. Assume further that 
$\alpha(x)$, $\beta(x)$ are bounded, Lipschitz non-negative functions defined on $\pa\Omega$ such that $\alpha(x)+\beta(x) \geq \sigma>0$ and  that
\begin{equation}\label{eq:n2}
\H^{n-2}\left(\partial \left\{\frac{\alpha}{\alpha+\beta}>\frac 1 2 \right\}\right)<\infty.
\end{equation}
Then the functional $ \J_\eps $ $\Gamma$-converges, when $\eps\to 0$ to 
$$ \F_{\frac{\alpha}{\alpha+\beta} }(E) = \frac 1 2 P(E,\Omega) + \int_{\pa \Omega } \min\left( \frac 1 2 ,\frac{\alpha(x)}{\alpha(x)+\beta(x)} \right) \chi_E(x) \, d\H^{n-1}(x)  .
$$

\end{theorem}

We note that for Dirichlet boundary conditions (and more generally when $\alpha\geq \beta$ on $\pa\Omega$), 
we have $ \F_{\frac{\alpha}{\alpha+\beta} } (E) = \frac 1 2 P(E)$ while with Neumann boundary conditions ($\alpha=0$) we find $ \F_{0 }(E) = \frac 1 2 P(E,\Omega)$.

\medskip

\subsection{Generalization: General elliptic operators}
It is relatively easy to generalize this result to more general elliptic operators in divergence form: given a continuous function $x\mapsto A(x)$ with $A(x)$ symmetric matrix satisfying
$$ \lambda I_n \leq A \leq \Lambda I_n$$
for $\lambda,\Lambda>0$. We can replace \eqref{eq:02}-\eqref{eq:bc0} with
\begin{equation}\label{eq:022}
\begin{cases}
 \phi -\eps^2\div(A \na \phi) = \chi_E \qquad & \mbox{ in } \Omega\\
\alpha \phi +\beta  \eps A \nabla \phi\cdot \frac{n}{\|n\|_A} = 0  \quad & \mbox{ on } \partial \Omega
\end{cases}
\end{equation}
with $\| x \|_{A} = \sqrt{x^T A x}$ for $x\in \R^n$. Then the same proofs (see Remarks \ref{rem:A1} and \ref{rem:A2}) show that the corresponding  functional $ \J_\eps $ $\Gamma$-converges, when $\eps\to 0$ to 
$$  \frac 1 2 P^A(E,\Omega) + \int_{\pa \Omega } \min\left( \frac 1 2 ,\frac{\alpha}{\alpha+\beta } \right) \chi_E(x) \| n\|_A \, d\H^{n-1}(x)  
$$
where the anisotropic perimeter is defined by
$$
P^A(E,\Omega) = \int_{\pa^* E} \| \nu_E(x) \|_ A d\H^{n-1}(x)
$$ 
or equivalently:
$$ P^A(E,\Omega)  = \int_\Omega \|D\chi_E\|_A
=\sup\left\{\int_{\Omega} \chi_E \, \mathrm{div}\, g\, dx \,;\, g\in [C^1_0(\Omega)]^n, \; |A^{-1}g(x)|\leq 1  \; \forall x \in\Omega\right\}.
$$
Similarly, in the nonlocal case $s\in(0,1)$, we can use the anisotropic  integro-differential operators 
$$
\mathcal L^{s/2}_A[\phi] = 
\mathrm{PV} \int_{\R^n}  \frac{u(x) - u(y)}{\|x-y\|_A^{n+s}} \, dy
$$
and similar results can be derived, with the $s$-perimeter replaced by the anisotropic fractional perimeter
$$
P_s^A(E) := \int_{\R^n} \int_{\R^n} \frac{\chi_E (x) \chi_{\C E}(y) }{\|x-y\|_A^{n+s}} \, dx \, dy = \frac 1 2  \int_{\R^n} \int_{\R^n}  \frac{|\chi_E(x) - \chi_E(y)|}{\|x-y\|_A^{n+s}} \, dx \, dy 
$$
and the potential  $\psi_F$ by $\psi^A_F (x): = \int_{F} \frac{1}{\|x-y\|_A^{n+s}}\, dy $.
 
\subsection{The case $s\in[1,2)$.}
When $\Omega = \R^n$ and $s\in(1,2)$, Theorem \ref{thm:per1} applies and gives 
$$ 
\lim_{\eps \to 0 } \J^s_\eps (E) = \sigma_{s,n} P(E)
$$
and the  $\Gamma$-limit is proved in \cite{BP}. When $\Omega$ is a subset of $\R^n$, the boundary conditions play a role.
We will not investigate this case in details, like we did for the case $s=2$ above. But for the sake of completeness, we prove that the scaling given in \eqref{eq:defJJ} is the correct one by proving the convergence of $\J_\eps^s(E)$ when  \eqref{eq:0} is supplemented with either Dirichlet or Neumann boundary conditions, that is when $\P$ solves
\eqref{eq:bc} with $\alpha\equiv0$ or $\beta\equiv0$.
We then have:
\begin{proposition}\label{prop:12}
Let $\Omega$ be a bounded open set with Lipschitz boundary and let $\J_\eps^s$ be defined by 
$$
\J^s_\eps(E):=
\begin{cases}
\displaystyle  \eps^{-1}|\ln\eps|^{-1}\int_\Omega \chi_E (1-\P)\, dx   & \mbox{ if } s=1\\[5pt]
\displaystyle  \eps^{-1}\int_\Omega \chi_E (1-\P)\, dx  & \mbox{ if } s\in(1,2)
  \end{cases}
  $$
where $\P$ solves \eqref{eq:bc} with either $\alpha\equiv0$ or $\beta\equiv0$.

There exists some positive constants $ \sigma^1_{n,s}$, $ \sigma^2_{n,s}$ (satisfying $\sigma^1_{n,1}= \sigma^2_{n,1}$) such that for any 
 set $E\subset \Omega$ such that $P(E,\Omega)<\infty$, we have 
\item[(i)] If $\beta\equiv0$ (Dirichlet boundary conditions):
$$\lim_{\eps \to 0 } \J^s_\eps (E) = \sigma^1_{n,s}  P(E,\Omega) + \sigma^2_{n,s} \int_{\pa \Omega} \chi_E  \, d \H^{n-1}
$$
\item[(i)] If $\alpha\equiv0$ (Neumann boundary conditions):
$$\lim_{\eps \to 0 } \J^s_\eps (E) = \sigma^1_{n,s}  P(E,\Omega) .
$$
\end{proposition}
We do not attempt to derive explicit formula for the constants $\sigma^1_{n,s}$ and $ \sigma^2_{n,s}$, except  
when $s=1$, where the proof gives
$$ \sigma^1_{n,1}= \sigma^2_{n,1}=\frac {c_{n,1}} 2 \int_{\partial B_1} |e\cdot y|d\H^{n-1}(y).$$
In particular, the limit with Dirichlet conditions is proportional to $P(E)$ when $s=1$.

\medskip

\subsection{An energy functional with competing local and nonlocal terms}
Finally, we go back to our initial motivation for studying this problem which arises in the context of a model for cell motility (see \cite{MW}) involving the functional 
$$ P(E) - \beta P_{K_\eps}(E) .$$
We note that the nonlocal perimeter appears now as a destabilizing term. When $\eps\ll1$ the results above show that the critical regime corresponds to $\beta \sim \eps^{-s}$ if $s\in(0,1)$ or $\beta \sim \eps^{-1}$ if $s\in (1,2]$.

We study this problem for $s\in(0,1)$ and for $s=2$. We fix $t>0$ and 
 define the functional $\G^s_\eps(E)$ by
\[
\G^s_\eps(E)=\displaystyle P(E,\Omega)- t \J_\eps^s(E)\qquad  \mbox{ when $s\in(0,1)$}
\]
and 
\[
\G_\eps(E)=\displaystyle P(E,\Omega)- t  \J_\eps(E) \qquad \mbox{ when $s=2$.}
\]

When $\Omega= \R^n$, we prove:
\begin{theorem}\label{thm:gammaRn}
Let $s\in(0,1)$ and $\Omega=\R^n$. Then the functional $ \G^s_\eps $ $\Gamma$-converges, when $\eps\to 0$ to 
$$
\G^s_0(E)=P(E)-tc_{n,s} P_s(E)
$$
for all $t>0$.
\end{theorem}
When $\Omega\neq \R^n$ and  $\P$ solves  \eqref{eq:bc}  (Robin boundary conditions), 
we have:
\begin{theorem}\label{thm:gammaOmega}
Let $\Omega$ be a  bounded subset of $\R^n$ such that $P(\Omega)<\infty$ and let $s\in (0,1)$. Assume further that $\alpha(x)$, $\beta(x)$ are bounded, Lipschitz non-negative functions such that $\alpha(x)+\beta(x) \geq \sigma>0.$
Then the functional $ \G^s_\eps $ $\Gamma$-converges, when $\eps\to 0$ to 
$$
\G^s_0(E)=P(E,\Omega)-tc_{n,s} \left[  P_s^L(E,\Omega) 
+ \int_{\C\Omega}  \frac{\alpha}{\alpha+\beta} \psi_{E }(x)\, dx   + \int_{\C\Omega}   \frac{\beta}{\alpha+\beta}	 \frac{ \psi_{E\cap\Omega} (x)\psi_{\C E \cap \Omega} (x)}{\psi_\Omega(x)} \, dx  \right]
$$
for all $t>0$.
\end{theorem}

When $s=2$, the problem is more delicate (in the case $s\in(0,1)$, the destabilizing term is of lower order than the stabilizing perimeter while when $s=2$, the destabilizing term is asymptotically of the same order).
We only consider the case of Neumann boundary conditions which is often the most relevant for applications (e.g. in the study of the cell motility model introduced in \cite{CMM}).

\begin{theorem}\label{thm:gammaOmega2} 
Let $\Omega$ be a bounded subset of $\R^n$ with $C^{1.\alpha}$ boundary and consider the functional
\[
\G_\eps(E)=\displaystyle P(E,\Omega)- t  \J_\eps(E), \qquad \J_\eps(E)= \eps^{-1}\int_\Omega \chi_E (1-\P)\, dx 
\]
where $\P$ solves the  local elliptic equation
\begin{equation}\label{eq:03}
\begin{cases}
 \phi -\eps^2\Delta \phi = \chi_E \qquad & \mbox{ in } \Omega\\
 \na \phi\cdot n = 0& \mbox{ on }\pa \Omega
 \end{cases}
\end{equation}
Then, for all $t\in(0,2)$, 
the functional $ \G_\eps $ $\Gamma$-converges, when $\eps\to 0$ to 
$$
\G_0(E)=\left(1-\frac t 2\right) P(E,\Omega)
$$
\end{theorem}

We do not carry out the detailed analysis of this limit with Robin boundary conditions, but we note that we have a similar issue as in Theorem \ref{thm:4}. For example when $\P$ solves the  local elliptic equation
$$
\begin{cases}
 \phi -\eps^2\Delta \phi = \chi_E \qquad & \mbox{ in } \Omega\\
\phi = 0& \mbox{ on }\pa \Omega,
 \end{cases}
$$
then Proposition \ref{prop:4} gives, for a fixed $E$,
$$  
\G_\eps(E)  = 
\displaystyle P(E,\Omega)- t  \J_\eps(E) \to \left(1-\frac t 2\right) P(E,\Omega) - t  \int_\Omega \chi_E(x) d\H^{n-1}(x).$$
However, this functional is not lower semicontinuous when $1-\frac t 2 < t$ (that is when $t>2/3$) so it cannot be the $\Gamma$-limit in that cases.
Instead, $\G_\eps$ $\Gamma$-converges to
$$
\G_0(E)=
\begin{cases}
\displaystyle \left(1-\frac t 2\right) P(E,\Omega) - t  \int_\Omega \chi_E(x) d\H^{n-1}(x) & \mbox{ if } 0\leq t\leq 2/3 \\[8pt]
\displaystyle \left(1-\frac t 2\right)\left(  P(E,\Omega) -  \int_\Omega \chi_E(x) d\H^{n-1}(x) \right) - \left(\frac {3 t} 2-1\right)  \H^{n-1}(\pa\Omega)& \mbox{ if }  2/3\leq t <2 .
\end{cases}
$$

 \subsection{Comments about the proofs}
 The proofs that we present in this paper rely strongly on the fact that $\P$ solves \eqref{eq:0} (with appropriate boundary conditions) rather than on the properties of the kernel $K_\eps$ solution of \eqref{eq:Keps0} (which are not so easy to determine near the boundary of $\Omega$).
The convergence of $\J_\eps^s(E)$ in the fractional case $s\in (0,1)$  (first part of Theorems \ref{thm:1} and \ref{thm:2}) is established using the relation between $P_s(E)$ and $(-\Delta)^{s/2}\chi_E$ (see \eqref{eq:LPs}). The $\liminf$ 
property requires a different formulation of $\J_\eps^s$ involving the lower semi-continuous norm $[\P]_{H^{s/2}(\Omega)}$, see \eqref{eq:Jepss0}.

When $s=2$, the convergence of $\J_\eps(E) $ (Proposition \ref{prop:4}) is established by writing
$$ \J_\eps(E) =\eps^{-1}\int_\Omega \chi_E (1-\P)\, dx=  -\eps \int_E \Delta \P \, dx  = - \eps \int_{\pa^* E } \na \P \cdot \nu_E(x)\, d\H^{n-1}(x) $$
and identifying the limit of $\eps  \na \P$ via a blow-up argument.
The $\liminf$ property cannot however be proved that way, and we rely instead on a different formulation, see \eqref{eq:Jepsrb} (which is reminiscent of the classical Modica-Mortola functional).
To complete the proof of Theorem \ref{thm:4}, we then need to establish the $\limsup$ property: While this is easily done by taking $E_\eps=E$ when 
 $\frac{\alpha}{\alpha+\beta} \leq  \frac 1 2 $ for all $x\in \pa\Omega$, a delicate construction is required when 
  $\frac{\alpha}{\alpha+\beta} >  \frac 1 2 $ on a non empty subset  of $\pa\Omega$.

The $\Gamma$-convergence of $\displaystyle P(E,\Omega)- t  \J^s_\eps(E)$ (Theorems \ref{thm:gammaRn}, \ref{thm:gammaOmega} and \ref{thm:gammaOmega2}) does not follow immediately from the work above since the role of the $\liminf$ and $\limsup$ are inverted by the minus sign. When $s\in(0,1)$ (Theorems \ref{thm:gammaRn} and \ref{thm:gammaOmega}), the positive term $P(E)$ is of higher order than the negative term, so we can establish the $\liminf$ property by first proving the boundedness of the sequence in $BV$.
The proof of Theorem \ref{thm:gammaOmega2} is  much more delicate since the two terms have the same order (asymptotically). It requires a precise estimate on the convergence of $\eps  \na \P$ (see Lemma \ref{lem:Pgrad}).

\section{Preliminary and notations}\label{sec:pre}
\subsection{Perimeter and reduced boundary}\label{sec:per}
Given $\Omega$ open set and $E\subset \R^n$ we recall that
\begin{align*}
P(E,\Omega) & = \int_{\Omega} |D\chi_E| = \sup\left\{\int_{\Omega} \chi_E \, \mathrm{div}\, g\, dx \,;\, g\in [C^1_0(\Omega)]^n, \; |g(x)|\leq 1  \; \forall x \in\Omega\right\}
\end{align*}
where $|g| = (g_1^2+\dots g_n^2)^{1/2}$.
We also have $P(E) := P(E,\R^n)$.
A Caccioppoli set $E$ is a set such that $P(E,\Omega)<\infty$ for all bounded open sets $\Omega$.
The {\it reduced boundary}, $\pa ^* E$ is defined as the set of $x\in \pa E$ where a notion of unit normal vector can be defined:
\begin{definition}\label{def:reduced}
A point $x\in \pa^* E$ if
\item[(1)] $\int_{B_\rho (x)} |D\chi_E| >0 $ for all $\rho>0$,
\item[(2)] $\nu_E(x) =\lim_{\rho\to 0 } \frac{\int_{B_\rho (x)} D\chi_E }{\int_{B_\rho (x)} |D\chi_E| }$ exists and $|\nu_E(x)| =1$.
\end{definition}
We then have (see \cite{giusti})
$$ P(E,\Omega)  = \H^{n-1}(\pa^* E\cap\Omega)  = \int_{\pa^* E\cap\Omega} d\H^{n-1}(x).$$
We also recall that for $x\in\pa^* E$ we have
$$
\lim_{\rho\to 0} \frac{1}{\rho^n}\int_{B_\rho(x) \cap\{\nu_E(x)\cdot(y-x)<0\}} \chi_E(y)\, dy = 0,\qquad 
\lim_{\rho\to 0}  \frac{1}{\rho^n}\int_{B_\rho(x) \cap\{\nu_E(x)\cdot(y-x)>0\}} (1-\chi_E(y))\, dy =0.
$$

\subsection{Fractional Laplacian and fractional perimeters}
The fractional Laplacian of a function $u(x)$ is defined by 
\begin{equation}\label{eq:deflaps}
(-\Delta)^{s/2} u (x): = c_{n,s}\mathrm{PV} \int_{\R^n}  \frac{u(x) - u(y)}{|x-y|^{n+s}} \, dy
\end{equation}
with $c_{n,s} = \frac{2^s \Gamma\left(\frac{n+s}{2} \right)}{\pi^{n/2} |\Gamma\left(-\frac s2 \right)|}$ and 
the fractional perimeter of a set $E$ is defined by
\begin{equation}\label{eq:defpers}
P_s(E) := \int_{\R^n} \int_{\R^n} \frac{\chi_E (x) \chi_{\C E}(y) }{|x-y|^{n+s}} \, dx \, dy = \frac 1 2  \int_{\R^n} \int_{\R^n}  \frac{|\chi_E(x) - \chi_E(y)|}{|x-y|^{n+s}} \, dx \, dy 
\end{equation}
which is also equal to the semi-norm $\frac 1 2 [\chi_E]_{W^{s,1}(\R^n)}$.
Given a set $\Omega\subset \R^n$, we also define the local contribution of the fractional perimeter by
\begin{equation}\label{eq:defpersL} 
P_s^L(E,\Omega) := \int_{\Omega} \int_{\Omega} \frac{ \chi_E (x) \chi_{\C E}(y)}{|x-y|^{n+s}} \, dx \, dy = \frac 1 2 [\chi_E]_{W^{s,1}(\Omega)}.
\end{equation}
We recall the following limits:
\begin{align*}
\lim_{s\to 1} (1-s) P_s(E)  = \omega_{n-1} P(E)\quad \mbox{ and } \quad \lim_{s\to 1} (1-s) P_s^L(E,\Omega) = \omega_{n-1} P(E,\Omega).
\end{align*}
For a given set  $E$, we define the potential
\begin{equation}\label{eq:psi} 
\psi_E (x): = \int_{\R^n} \frac{\chi_E(y)}{|x-y|^{n+s}}\, dy \qquad x\in \C E.
\end{equation}
We then have
\begin{lemma} \label{lem:Psp}
If $P_s(E)<\infty$, then $\psi_E \in L^1(\C E)$ and 
$$P_s(E) = \int_{\R^n} \chi_{\C E}(x) \psi_E(x)\, dx  =  \int_{\R^n} \chi_E(x) \psi_{\C E}(x)\, dx .$$
Given $\Omega$ such that $P_s(\Omega)<\infty$ and $E$ such that $P_s^L(E,\Omega)<\infty$, we have
$$
P_s^L(E,\Omega) =  \int_{\R^n}  \chi_{\C E\cap\Omega}(x)\psi_{E\cap\Omega} (x)\, dx= \int_{\R^n} \chi_{E\cap\Omega} (x) \psi_{\C E\cap\Omega}(x)\, dx .
$$
\end{lemma}
\begin{proof}
The sequence   $v_\eta(x) = \int_{\R^n\setminus B_\eta(x)} \frac{\chi_E(y)}{|x-y|^{n+s}}\, dy$ is positive, increasing as $\eta\to0$ and satisfies $\int_{\C E} v^\eta(x)\, dx\leq P_s(E)$. By Beppo-Levi Lemma, we can thus define $ \psi_E (x) = \lim_{\eta\to0} v_\eta(x)$ and the result follows since $P_s(E) = \lim_{\eta\to0} \int_{\R^n} \chi_{\C E}(x) v_\eta(x)\, dx$. 
\end{proof}

With the notation \eqref{eq:psi}, we can also write
\begin{align}
 (-\Delta)^{s/2}  \chi_E 
 & =c_{n,s} \lim_{\eta\to0} \int_{\R^n\setminus B_\eta(x)} \frac{\chi_E(x)-\chi_E(y)}{|x-y|^{n+s}}\, dy  \nonumber   \\
 & =   c_{n,s}\lim_{\eta\to0} \left[ \chi_E(x) \int_{\R^n\setminus B_\eta(x)} \frac{1-\chi_E(y)}{|x-y|^{n+s}}\, dy
 + (1-\chi_E(x))  \int_{\R^n\setminus B_\eta(x)} \frac{-\chi_E(y)}{|x-y|^{n+s}}\, dy\right]\nonumber \\
& = c_{n,s}( \chi_E \psi_{\C E} - \chi_{\C E} \psi_E ) .\label{eq:LPs}
\end{align}
In particular, $P_s(E)<\infty$ implies that $ (-\Delta)^{s/2}  \chi_E $ is in $L^1(\R^n)$.
\medskip

Finally, we recall the following interpolation inequalities:
\begin{proposition}
For all $s\in(0,1)$,
\begin{equation}\label{eq:PsP}
 P_s(E) \leq  \frac{n\omega_n 2^{-s}}{s(1-s)} P(E)^s |E|^{1-s} \qquad \mbox{ for all set $E$ with finite perimeter.}
 \end{equation}
Given $\Omega\subset \R^n$, for all $s\in(0,1)$, 
\begin{equation}\label{eq:PsLP}
 P_s^L(E,\Omega) \leq \frac{n\omega_n 2^{-s}}{s(1-s)} P(E,\mathrm {Conv}(\Omega))^s |E\cap\Omega|^{1-s} \qquad \mbox{ for all set $E$ with finite perimeter}
 \end{equation}
where $\mathrm {Conv}(\Omega)$ denotes the convex hull of $\Omega$ and $\omega_n$ denotes the volume of $n$-dimensional unit ball.
\end{proposition}
The first inequality is classical, but we provide a proof of the second one in the appendix for the sake of completeness.

\subsection{Alternative formula for $\J_\eps^s$}
We end this section by deriving an alternative formula for $\J_\eps^s$ which is useful in the proofs.
\begin{proposition}\label{prop:alt}
When $s<2$ the functional $\J^s_\eps$ defined by \eqref{eq:defJJ} can also be written as:
\begin{equation}\label{eq:Jepss0}
\J_\eps^s(E) 
= 
\begin{cases}
\displaystyle \eps^{-s}\int_{\R^n} (\chi_E-\P)^2 \, dx  +  \frac {c_{n,s}} 2  [\P]_{H^{s/2}(\R^n)}^2 & \mbox{ if } \Omega = \R^n\\[5pt]
\displaystyle \eps^{-s}\int_\Omega (\chi_E-\P )^2 \, dx  + \frac { c_{n,s}} 2  \int_{\R^{2n}\setminus (\C\Omega)^2}  \frac{|\P(x) - \P(y)|^2}{|x-y|^{n+s}} \, dx \, dy - \int_{\C\Omega} \P \N(\P)\, dx &\mbox{ if } \Omega \neq \R^n.\end{cases}
 \end{equation} 
And similarly when $s=2$: 
\begin{equation}\label{eq:Jepsrb}
\J_\eps(E) =
\begin{cases}
\displaystyle \eps^{-1}\int_{\R^n}  |\chi_E-  \P |^2\, dx + \eps \int_{\R^n} |\na \P| ^2\, dx  & \mbox{ if } \Omega = \R^n\\[5pt]
\displaystyle \eps^{-1}\int_\Omega |\chi_E-  \P |^2\, dx + \eps  \int_\Omega |\na \P| ^2\, dx - \eps\int_{\pa \Omega} \P \na \P \cdot n \, d\H^{n-1}(x) &\mbox{ if } \Omega \neq \R^n.
 \end{cases}
\end{equation}
\end{proposition}
Note that these formulas are reminiscent of fractional and classical Modica-Mortola functional for which the $\Gamma$-convergence is studied in particular in \cite{SV12}.
\begin{proof}
First, we make the following simple computation:
Let $u\in L^1\cap L^2$ (we will later take $u=\chi_{E}$) and let $w$ (which will be $\phi^\eps_{E}$) denote the solution of $w+\eps^s (-\Delta)^{s/2} w = u$ in $\R^n$.
Multiplying the equation by $w$ and integrating, we get:
$$ \eps^{-s}\int_\Omega (w^2 -wu)\, dx  + \int_\Omega w   (-\Delta)^{s/2} w\, dx =0$$
and we can use this equality to write
\begin{align*}
\eps^{-s}\int_{\Omega } u (1-w)\, dx 
& = \eps^{-s}\int _{\Omega}u (1-u) + \eps^{-s}\int_{\Omega} (u^2 -uw)  +\eps^{-s}\int _{\Omega} (w^2 -wu)\, dx  + \int_{\Omega} w   (-\Delta)^{s/2} w\, dx\nonumber \\
& = \eps^{-s}\int_{\Omega} u (1-u) + \eps^{-s}\int _{\Omega}(u-w)^2 \, dx  
+ \int_{\Omega} w   (-\Delta)^{s/2} w\, dx.
\end{align*}
Taking $u=\chi_E \in L^1(\Omega)$, we deduce 
$$
\J_\eps^s (E) = 
\eps^{-s}\int_{\Omega} \chi_E (1-\P)\, dx  =    \eps^{-s}\int _{\Omega}(\chi_E-\P)^2 \, dx  
+ \int_{\Omega} \P  (-\Delta)^{s/2} \P \, dx.
$$
Then the formula \eqref{eq:Jepss0} follows from the definition of $(-\Delta)^{s/2} $.

A similar computation when $s=2$ yields
$$  \J_\eps (E) =  \eps^{-1}\int_{\Omega} \chi_E (1-\P)\, dx
 = \eps^{-1}\int_\Omega |\chi_E-  \P |^2\, dx + \eps \int_\Omega  \P  (-\Delta) \P \, dx
 $$
and an integration by parts gives \eqref{eq:Jepsrb}.


\end{proof}

\begin{remark}
We can relax the definition of $\J_\eps$ to non-negative $\BV$-functions by setting
$$ 
\J^s_\eps (u) = \eps^{-s}\int_{\R^n} u (1-w^\eps)\, dx  .
$$
The computation above yields
$$
\J^s_\eps (u) = \eps^{-s}\int_{\R^n} u (1-u) + \eps^{-s}\int _{\R^n}(u-w^\eps)^2 \, dx  + \frac {c_{n,s}} 2 
 [w^\eps]_{H^{s/2}(\R^n)}^2.
 $$
and we see that the first term diverges unless $u(x) \in\{0,1\}$ for all $x$. However, $u\mapsto u(1-u)$ is not a double-well potential, unless we add the constraint that $0\leq u\leq 1$. 
While the non-negativity of $u$ is natural in many context, the upper bound would have to be imposed by some over-crowding prevention mechanisms.
We are led to the energy functional:
$$
\bar \J^s_\eps (u) = 
\begin{cases}
 \J^s_\eps (u) & \mbox { if } 0\leq u\leq 1 \\
 \infty & \mbox{ otherwise.}
\end{cases}
$$
This is very similar to the relaxation of the perimeter functional with the heat content energy approximation
used for example in \cite{EsedogluOtto,LauxOtto,JacobsKimMeszaros} to 
 construct weak solutions (via minimizing movements schemes) of 
 multi-phase mean curvature flow \cite{EsedogluOtto,LauxOtto} or 
 the Muskat problem with surface tension \cite{JacobsKimMeszaros}.

Note also that the energy $\bar \J^s_\eps $ appears naturally in the  incompressible limit ($m\to\infty$) of the following Keller-Segel model for Chemotaxis (or congested aggregation, see \cite{CraigKimYao}):
$$
\begin{cases}
\pa_t  u = \Delta u^m -\beta \div(u \na \phi) \\
\phi + \eps^s(-\Delta)^{s/2} \phi = u .
\end{cases}
$$
This provides another motivation for studying this problem in a bounded domain with appropriate boundary condition on $\phi$.
\end{remark}

\section{Proofs of the theorems for  $s\in(0,1)$}
\subsection{The case $\Omega =\R^n$ - Proof of Theorem \ref{thm:1}}
We start with the following lemma:
\begin{lemma}\label{lem:1}
Let $E$ be a measurable set.
The function $\P$ solution of \eqref{eq:P0R} satisfies
\item[(i)] $0\leq \P(x)\leq 1$ in $\R^n$
\item[(ii)] Up to  a subsequence $\P$ converges weakly in $L^q_{loc}(\R^n)$ to $\chi_{E}$  for all $q\in(1,\infty)$
\item[(iii)] Up to  a subsequence $\P$ converges strongly in $L^1_{loc}(\R^n)$ and almost everywhere to $\chi_{E}$ 
\end{lemma}
\begin{proof}
The maximum principle gives (i) and 
since  $\P$ clearly converges to $\chi_E$ in the sense of distribution, (ii) follows.
To prove (iii), we first write, for any ball $B_R$ (using (i)):
\begin{align*}
\int_{B_R} |\P-\chi_E|\, dx
& =  \int_{B_R} |(\P -1) \chi_E +\P  (1 -\chi_E)|\, dx \\
& \leq  \int_{B_R} (1-\P) \chi_E \, dx  + \int_{B_R} \P  (1 -\chi_E) \, dx 
\end{align*}
and note that (ii) implies that the right hand side converges to $0$. It follows that $\P$ converges strongly in $L^1_{loc}$ to $\chi_E$ and (up to another subsequence) we can assume that it converges almost everywhere.
\end{proof}
Note that we can write $\P(x) = K_\eps \star \chi_E$ and $K_\eps>0$ is an approximation of unity, so the results of the lemma above are obvious. However, the proof we gave above will be easy to carry out in the other settings presented in this paper.

\begin{proof}[Proof of Theorem \ref{thm:1}]
To prove (i), we write (using \eqref{eq:P0R}):
$$ 
\J^s_\eps(E)=  \eps^{-s}\int_{\R^n} \chi_E( 1-\P)\, dx 
=  \int_{\R^n} \chi_E (-\Delta)^{s/2} \P \, dx
=  \int_{\R^n} \P  (-\Delta)^{s/2}  \chi_E\, dx
$$
and using  \eqref{eq:LPs}, we get:
\begin{equation}\label{eq:int1}
\J^s_\eps(E) = c_{n,s}  \int_{\R^n} \P (\chi_E \psi_{\C E} - \chi_{\C E} \psi_E)\, dx.
\end{equation}
Lemma \ref{lem:1}-(i) together with
Lemma \ref{lem:Psp}
 imply that $|\J^s_\eps(E)|\leq c_{n,s} P_s(E)$ so we can take a subsequence (still denoted $\eps$) such that $\J_\eps(E)$ converges. We will prove that the limit of that subsequence must be  $c_{n,s}   P_s(E)$ which implies the result.

To prove this, we note that $|\P (\chi_E \psi_{\C E} - \chi_{\C E} \psi_E)|\leq \chi_E \psi_{\C E} + \chi_{\C E} \psi_E$
and the condition $P_s(E)<\infty$ implies that $\chi_E \psi_{\C  E}$ and $\chi_{\C E} \psi_E$ are in $L^1(\R^n)$.
We can thus pass to the limit in \eqref{eq:int1} using Lebesgue dominated convergence theorem and Lemma \ref{lem:1} (iii).
We deduce (up to another subsequence)
$$ \J^s_\eps(E)  \to 
c_{n,s}  \int_{\R^n} \chi_{E} (\chi_E \psi_{\C E} - \chi_{\C E} \psi_E)\, dx=
c_{n,s} \int_{\R^n}\chi_{E} \psi_{\C E}\, dx = c_{n,s} P_s(E)
$$
and (i) follows.
\medskip

In order to prove the second part of Theorem \eqref{thm:1} (the $\liminf$ statement), we use the formula \eqref{eq:Jepss0}:
\begin{equation}\label{eq:Jepss} 
\J^s_\eps(E) =
 \eps^{-s}\int_{\R^n} (\chi_E-\P)^2 \, dx  +  \frac {c_{n,s}} 2 
 [\P]_{H^{s/2}(\R^n)}^2.
\end{equation}
Since the result trivially holds when $\liminf_{\eps\to0} \J^s_\eps(E_\eps)=\infty$, we can assume that $\liminf_{\eps\to0} \J^s_\eps(E_\eps)<\infty$ and consider a subsequence (still denoted $\eps$) along which $\J^s_\eps(E_\eps)$ is bounded.
Equality \eqref{eq:Jepss} then implies that $\phi^\eps_{E_\eps}$ converges strongly in $L^2$ to $\chi_E$ (recall that  $\chi_{E_\eps}$ converges to $\chi_E$ in $L^1$ and thus also in $L^2$) and
the lower semicontinuity of the $H^{s/2}$ norm gives
\begin{align*}
\liminf_{\eps\to 0} \J^s_\eps(E_\eps) 
\geq \liminf_{\eps\to 0}  \frac {c_{n,s}} 2  \left[\phi^\eps_{E_\eps} \right]_{H^{s/2}(\R^n)}^2 & \geq  \frac {c_{n,s}} 2  \left[\chi_E \right]_{H^{s/2}(\R^n)}^2\\
&=  \frac {c_{n,s}} 2  \int_{\R^n} \int_{\R^n}  \frac{|\chi_E(x) - \chi_E(y)|^2}{|x-y|^{n+s}} \, dx \, dy  \\
&=  \frac {c_{n,s}} 2  \int_{\R^n} \int_{\R^n}  \frac{|\chi_E(x) - \chi_E(y)|}{|x-y|^{n+s}} \, dx \, dy  =  {c_{n,s}} P_s(E)  \\
\end{align*}
which completes the proof.
\end{proof}

\subsection{The case $\Omega\neq \R^n$ - Proof of Theorem \ref{thm:2}}
We assume that $\Omega$ is a bounded subset of $\R^n$ and that $\P$ is the solution of
\begin{equation}\label{eq:phi1}
\begin{cases}
 \phi + \eps^s (-\Delta)^{s/2} \phi = \chi_E \qquad & \mbox{ in } \Omega\\
\alpha \phi(x) + \beta \widetilde \N (\phi)(x) = 0 \qquad & \mbox{ in } \C\Omega
\end{cases}
\end{equation}
with $\alpha$, $\beta :\C\Omega \to [0,\infty]$ satisfying $\alpha(x)+\beta(x)>0$.
These Robin boundary conditions satisfy the maximum principle, so we can prove (the proof is similar to that of Lemma \ref{lem:1}):
\begin{lemma}\label{lem:2}
Let $E$ be a measurable set.
The function $\P $ satisfies
\item[(i)] $0\leq \P(x) \leq 1$ in $\R^n$
\item[(ii)] Up to  a subsequence $\P$ converges weakly in $L^q(\Omega)$ to $\chi_{E}$  for all $q\in(1,\infty)$.
\item[(iii)] Up to  a subsequence $\P$ converges strongly in $L^1(\Omega)$ and almost everywhere to $\chi_{ E}$ 
\end{lemma}

\begin{proof}[Proof of Theorem \ref{thm:2}-(i)]
We recall that $E\subset \Omega$.
Proceeding as in the previous section, we write
\begin{align}
\J^s_\eps(E) =  \eps^{-s}\int_\Omega \chi_E (1-\P)  \, dx 
&  = \int_\Omega \chi_E  (-\Delta)^{s/2} \P \, dx\nonumber \\
& =  \int_\Omega \P  (-\Delta)^{s/2}  \chi_E \, dx -\int_{\C\Omega} \chi_E \N(\P)- \P \N(\chi_E)\, dx\nonumber \\
& =  \int_\Omega \P  (-\Delta)^{s/2}  \chi_E \, dx + \int_{\C\Omega} \P \N(\chi_E)\, dx \label{eq:introbin}
\end{align}
(recall that $E\subset \Omega$).
We note that the conditions $P_s^L(E,\Omega)<\infty$, $P_s(\Omega)<\infty$ together with Lemma \ref{eq:psi}  imply that
$\chi_{\C E} \psi_E \in L^1(\Omega)$ and $\chi_E \psi_{\C E} = \chi_E \psi_{\C E\cap\Omega } +\chi_E \psi_{\C \Omega}\leq \chi_E \psi_{\C E\cap\Omega } +\chi_\Omega \psi_{\C \Omega} \in L^1(\Omega)$. We can thus proceed as in the case $\Omega = \R^n$ with the first term in \eqref{eq:introbin}: 
Using  \eqref{eq:LPs}, we write
\begin{align}
 \int_\Omega \P  (-\Delta)^{s/2}  \chi_E \, dx 
& = c_{n,s}  \int_{\Omega} \P  (\chi_E \psi_{\C E} - \chi_{\C E} \psi_E)\, dx\nonumber  \\
& \longrightarrow c_{n,s}  \int_{\Omega} \chi_{E} \psi_{\C E}\, dx
= c_{n,s}  \int_{\R^n} \chi_{E} \psi_{\C E}\, dx= c_{n,s}  \int_{\R^n} \chi_{\C E} \psi_{E}\, dx.
\label{eq:lim1}
\end{align}

For the second term, we note that $ \N(\chi_E) =-c_{n,s} \psi_{E}$ for $x\in \C \Omega$,
 and using the boundary condition, we find (note that $\widetilde \N (\phi ) = \phi(x) - \frac{1}{\psi_\Omega(x)} \int_{\Omega} \frac{\phi(y)}{|x-y|^{n+s}} \, dy$):
 \begin{equation}\label{eq:bc2}
\P  (x) = \frac{\beta}{\alpha+\beta}	 \frac{1 }{\psi_\Omega(x)} \int_{\Omega} \frac{\P(y)}{|x-y|^{n+s}} \, dy\qquad x\in \R^n\setminus \Omega.
\end{equation}
Hence
\begin{align*}
 \int_{\C\Omega} \P \N(\chi_E)\, dx
 & =  - c_{n,s}\int_{\C\Omega}  \int_{\Omega}
 \frac{\beta(x)}{\alpha(x)+\beta(x)}	 \frac{\psi_{E}(x) }{\psi_\Omega(x)}  \frac{\P(y)}{|x-y|^{n+s}} \, dy
  dx.
 \end{align*}
We can pass to the limit using Lebesgue dominated convergence theorem since 
$$\left| \int_{\C\Omega}  
 \frac{\beta(x)}{\alpha(x)+\beta(x)}	 \frac{\psi_{E}(x) }{\psi_\Omega(x)}  \frac{ 1}{|x-y|^{n+s}} 
  dx\right| \leq \frac{\psi_{E}(x)\psi_{\C\Omega} (x)}{\psi_\Omega(x)}  \leq \psi_{\C\Omega}(x)\in L^1(\Omega)  $$
(recall that $P_s(\Omega)<\infty$)
and we get
\begin{align}
\lim_{\eps\to 0} \int_{\C\Omega}\P \N(\chi_E)\, dx
& =   - c_{n,s}\int_{\C\Omega}  \int_{\Omega}
 \frac{\beta(x)}{\alpha(x)+\beta(x)}	 \frac{\psi_{E}(x) }{\psi_\Omega(x)}  \frac{\chi_E(y)}{|x-y|^{n+s}} \, dy
  dx\nonumber\\
  & =   - c_{n,s}\int_{\C\Omega}  
 \frac{\beta(x)}{\alpha(x)+\beta(x)}	 \frac{\psi_{E}(x)^2 }{\psi_\Omega(x)}   dx. \label{eq:lim2}
\end{align}
Putting together \eqref{eq:introbin}, \eqref{eq:lim1} and \eqref{eq:lim2}, we deduce
 \begin{align}
 \lim_{\eps\to 0} \J^s_\eps & = 
  c_{n,s} \left[ \int_{\R^n} \chi_{\C E} \psi_{E}\, dx
   - \int_{\C\Omega}  
 \frac{\beta}{\alpha+\beta}	 \frac{\psi_{E}(x)^2 }{\psi_\Omega(x)}   dx\right]\nonumber\\
 & = 
  c_{n,s} \left[ \int_{\Omega} \chi_{\C E} \psi_{E}\, dx 
   +\int_{\C\Omega} ( \psi_{E}
 -  \frac{\beta}{\alpha+\beta}	 \frac{\psi_{E}(x)^2 }{\psi_\Omega(x)} )  dx\right]\label{same}\\
& =  c_{n,s} \left[  P_s^L(E,\Omega) 
+ \int_{\C\Omega}  \frac{\alpha}{\alpha+\beta} \psi_{E}(x)\, dx   + \int_{\C\Omega}   \frac{\beta}{\alpha+\beta}	 \frac{ \psi_{E} (x)(\psi_{\Omega} (x)-\psi_E(x))}{\psi_\Omega(x)} \, dx  \right]\nonumber
\end{align}
and the result follows.
\end{proof}

\begin{proof}[Proof of Theorem \ref{thm:2}-(ii)]
For any set $E$,
the boundary condition  gives
$$ 
-\P \N(\P) = c_{n,s} \psi_\Omega\left(  \frac{\alpha}{\alpha+\beta} |\P|^2 + \frac{\beta}{\alpha+\beta}| \widetilde \N(\P)|^2\right)$$
so the formula \eqref{eq:Jepss0} implies
\begin{equation}\label{eq:JepssO}
\begin{aligned}
\J^s_\eps(E) &  = \eps^{-s}\int_\Omega (\chi_E-\P)^2 \, dx  
+  \frac  { c_{n,s}}  2  \int_{\R^{2n}\setminus (\C\Omega)^2}  \frac{|\P(x) - \P(y)|^2}{|x-y|^{n+s}} \, dx \, dy\\
&\qquad \qquad + c_{n,s} 
\int_{\C\Omega} \psi_\Omega\left[ \frac{\alpha}{\alpha+\beta} |\P|^2 + \frac{\beta}{\alpha+\beta}| \widetilde \N(\P)|^2\right]  \, dx.
\end{aligned}
 \end{equation}

We can now complete the proof of Theorem \ref{thm:2} by proceeding as in the proof of Theorem \ref{thm:1}:
Let  $\{E_\eps\}_{\eps>0}$ be a sequence of sets which converges to $E$ in $L^1$ (and in $L^2$)  such that $\J^s_\eps(E_\eps)$ is bounded.
Equality \eqref{eq:JepssO} thus implies that $\phi^\eps_{E_\eps}$ converges strongly in $L^2(\Omega)$ to $\chi_E$.
Following \cite{DRV}[Proposition 3.1], we note that the space $W$ equipped with the norm
$$\|u \|_{W}^2 = \int_\Omega u^2\, dx + \int_{\C\Omega}  \frac{\alpha}{\alpha+\beta}\psi_\Omega u^2\, dx
+  \int_{\R^{2n}\setminus (\C\Omega)^2}  \frac{|u(x) - u(y)|^2}{|x-y|^{n+s}} \, dx \, dy$$
is a Hilbert space (note that $\psi_\Omega \in L^1(\C\Omega)$ since $P_s(\Omega)<\infty$)
and \eqref{eq:JepssO} implies that $\phi^\eps_{E_\eps}$ is bounded in $W$ and thus weakly converges to $\phi_0$ in that space. 
We already know that $\phi_0(x)=\chi_E(x)$ for $x\in \Omega$, and $\phi_0$ can be identified in $\C\Omega$ by taking the limit in \eqref{eq:bc2} (in $ \mathcal D'(\Omega)$), showing that  
$$\phi_0(x) = \frac{\beta}{\alpha+\beta} \frac{\psi_E}{\psi_\Omega} \mbox{ a.e. in } \C\Omega.$$
Similarly, $ \widetilde \N(\P)$ converges weakly in $L^2(\C\Omega , \frac{\beta}{\alpha+\beta}\psi_\Omega\, dx) $
to$-\frac{\alpha}{\alpha+\beta} \frac{\psi_E}{\psi_\Omega}$ (the limit can be easily identified using the exterior boundary condition).
The lower semicontinuity of the norms then gives
\begin{align*}
\liminf_{\eps\to 0} \J^s_\eps(E_\eps) 
& \geq 
 \frac  { c_{n,s}}  2  \int_{\R^{2n}\setminus (\C\Omega)^2}  \frac{|\phi_0(x) - \phi_0(y)|^2}{|x-y|^{n+s}} \, dx \, dy\\
&\qquad \qquad + c_{n,s} 
\int_{\C\Omega} \psi_\Omega\left[ \frac{\alpha}{\alpha+\beta} \left|\frac{\beta}{\alpha+\beta} \frac{\psi_E}{\psi_\Omega}\right|^2 + \frac{\beta}{\alpha+\beta}\left| \frac{\alpha}{\alpha+\beta} \frac{\psi_E}{\psi_\Omega} \right|^2\right]  \, dx\\
& =
 \frac  { c_{n,s}}  2  \int_{\Omega^2}  \frac{|\phi_0(x) - \phi_0(y)|^2}{|x-y|^{n+s}} \, dx \, dy
+  c_{n,s}  \int_{\Omega}\int_{\C\Omega}  \frac{|\phi_0(x) - \phi_0(y)|^2}{|x-y|^{n+s}} \, dx \, dy
 \\
&\qquad \qquad + c_{n,s} 
\int_{\C\Omega} \frac{\alpha\beta}{(\alpha+\beta)^2}  \frac{\psi_E^2}{\psi_\Omega} \, dx.
\end{align*}
Finally we have
$$   \frac 1 2   \int_{\Omega^2}  \frac{|\phi_0(x) - \phi_0(y)|^2}{|x-y|^{n+s}} \, dx \, dy
= \frac 1 2   \int_{\Omega^2}  \frac{|\chi_E(x) -\chi_E(y)|}{|x-y|^{n+s}} \, dx \, dy = P_s^L(E,\Omega)
$$
and 
a simple computation gives
\begin{align*}
\int_{ \Omega}\int_{\C\Omega}  \frac{|\phi_0(x) - \phi_0(y)|^2}{|x-y|^{n+s}} \, dx \, dy
& = 
\int_{\C \Omega}\int_{\Omega}  \frac{\phi_0(x)^2 - 2 \phi_0(x)\chi_E(y) +\chi_E (y)^2}{|x-y|^{n+s}} \, dy \, dx\\
& = 
\int_{\C \Omega} (\phi_0(x)^2 \psi_\Omega - 2 \phi_0(x)\psi_E(x) +\psi_E (x)) \, dx\\
& = 
\int_{\C \Omega} \left(\frac{(-2\alpha\beta-\beta^2)}{(\alpha+\beta)^2}\frac{\psi_E^2}{\psi_\Omega} +\psi_E (x) \right)\, dx.
\end{align*}
We deduce
\begin{align*}
\liminf_{\eps\to 0} \J^s_\eps(E_\eps) 
& \geq 
 c_{n,s}P_s^L(E,\Omega)
 + c_{n,s} 
\int_{\C\Omega}  \left(\frac{(-2\alpha\beta-\beta^2)}{(\alpha+\beta)^2}\frac{\psi_E^2}{\psi_\Omega} +\psi_E (x) +\frac{\alpha\beta}{(\alpha+\beta)^2}  \frac{\psi_E^2}{\psi_\Omega}  \right) \, dx.\\
& \geq 
 c_{n,s}P_s^L(E,\Omega)
 + 
 c_{n,s} 
\int_{\C\Omega} \left(\frac{ -\beta}{\alpha+\beta}  \frac{\psi_E^2}{\psi_\Omega} +\psi_E (x)\right)\, dx
 \end{align*}
 which completes the proof since $\frac{ -\beta}{\alpha+\beta}  \frac{\psi_E^2}{\psi_\Omega} +\psi_E (x) =  \frac{\alpha}{\alpha+\beta} \psi_{E}(x)   +  \frac{\beta}{\alpha+\beta}	 \frac{ \psi_{E} (x)(\psi_{\Omega} (x)-\psi_E(x))}{\psi_\Omega(x)} $.
\end{proof}

\section{The case $s=2$ - Proof of Theorem \ref{thm:4}}
When $s=2$, the function $\P$ solves the boundary value problem:
\begin{equation}\label{eq:phirob}
\begin{cases}
 \P  -\eps^2 \Delta\P = \chi_{ E}  & \mbox{ in } \Omega \\
\alpha  \P +\eps \beta \nabla  \P \cdot n = 0  & \mbox{ on } \partial \Omega.
\end{cases}
\end{equation}
and we have
$$
\J_\eps(E) = \eps^{-1} \int_\Omega  \chi_E (1-\P)\, dx =  -\eps \int_E \Delta \P \, dx  .
$$
The proof of Theorem \ref{thm:4} is very different, and much more delicate than the case $s\in(0,1)$.
We start with the proof of Proposition \ref{prop:4} which identifies the limit of $\J_\eps(E)$ for a fixed set $E$.

\begin{proof}[Proof of Proposition \ref{prop:4}]
For a set $E$ such that $P(E)<\infty$, we write
\begin{align*}
\J_\eps(E)& =  -\eps \int_E \Delta \P \, dx  \\
& = - \eps \int_{\pa^* E } \na \P \cdot \nu_E(x)\, d\H^{n-1}(x) 
\end{align*}
where 
$\pa^* E$ is the reduced boundary of $E$ and 
$\nu_E$ is the outward pointing unit normal vector (see Definition \ref{def:reduced}).
To pass to the limit, we take $x_0\in \pa^* E$ and define 
the function $w^\eps (x) = \P (x_0+\eps x)$, which solves
$$ 
\begin{cases}
 w^\eps -\Delta w^\eps = \chi_{ E}(x_0+\eps x) & \mbox{ in } \Omega_\eps\\
\alpha^\eps  w^\eps + \beta^\eps \nabla  w^\eps \cdot n = 0  & \mbox{ on } \partial \Omega_\eps.
\end{cases}
$$
with $\alpha^\eps= \alpha(x_0+\eps x)$, $\beta^\eps= \beta(x_0+\eps x)$ and $\Omega_\eps=\eps^{-1}(\Omega -x_0)$.
It is readily seen that $0\leq  w^ \eps\leq 1$ (maximum principle) and  that 
$w^\eps\in C^{1,\gamma}$ so that
$\eps \na\P (x_0)= \na  w^\eps(0)$ is well defined.
We conclude thanks to the following lemma:
\begin{lemma}\label{lem:3}
There exists a constant $C$ such that $|\eps \na\P (x_0)|=|\na w^\eps (0) |\leq C$ for all $x_0$.
Furthermore, 
\item If $x_0\in \pa^* E\cap \Omega$, then 
$$ 
\eps \na\P (x_0) \to  - \frac 1 2  \nu_E(x_0) 
$$
\item If $x_0\in  \pa^* E\cap \pa\Omega$ and $\pa\Omega$ is differentiable at $x_0$, then $\nu_E(x)=n(x)$ and 
$$ 
\eps \na\P (x_0) \to - \frac{\alpha}{\alpha+\beta} n(x_0)
$$
\end{lemma}
This lemma (together with Lebesgue dominated convergence theorem) implies
\begin{align*}
\lim_{\eps\to 0} \J_\eps(E) 
& = \frac 1 2 \int_{\pa^* E \cap \Omega} \,  d\H^{n-1}(x) +  \int_{\pa^* E\cap \pa\Omega}  \frac{\alpha}{\alpha+\beta}  \, d\H^{n-1}(x) \\
& = \frac 1 2 \H^{n-1} (\pa^* E \cap \Omega) +  \int_{\pa\Omega} \chi_E \frac{\alpha}{\alpha+\beta}  \, d\H^{n-1}(x) 
\end{align*}
and the result follows since $ \H^{n-1} (\pa^* E \cap \Omega) =P(E,\Omega)$.
\end{proof}

\begin{proof}[Lemma \ref{lem:3}]

If $x_0\in \Omega$, then by definition of the reduced boundary, we have $\chi_{ E}(x_0+\eps x) \to \chi_{\{x\cdot \nu_E(x_0)<0\}} $ in $L^1_{loc}$ (\cite{giusti}, Theorem 3.8).
Furthermore,  for all $R>0$ we have $B_R(0)\subset \Omega_\eps$ for $\eps$ small enough
so standard elliptic regularity theory (using the fact that $|w^\eps|\leq 1$) implies that $w^\eps$ is bounded in $C^{1,\gamma}(B_R)$. In  particular, $w^\eps$
(resp. $\na w^\eps$)
converges locally uniformly  to $w_0$ (resp. $\na w_0$) unique bounded solution of 
$$ 
w_0 -\Delta w_0 = \chi_{\{x\cdot \nu_E(x_0)<0\}} \mbox{ in } \R^n. 
$$
This solution is of the form $w_0(x) = \vphi(x\cdot \nu_E(x_0))$ where $\vphi:\R\to[0,1]$ solves
$ \vphi-\vphi'' = \chi_{\{x<0\}}$.
We easily find 
$$
\vphi(x) = \begin{cases}
1-\frac 1 2 e^{x} & \mbox{ for } x<0\\
\frac 1 2 e^{-x} & \mbox{ for } x>0\
\end{cases}$$
and so $\na w_0  = \vphi'(0) \nu_E(x_0) =-\frac 1 2 \nu_E(x_0)$.

\medskip

If $x_0\in\pa \Omega$, we have $\nu_E(x_0)=n(x_0)$ and (as above)
 $\chi_{E}(x_0+\eps x) \to \chi_{\{x\cdot n(x_0)<0\}}        $ in $L^1_{loc}$. 
 Furthermore, since $\alpha_\eps$, $\beta_\eps$ are bounded and Lipschitz uniformly in $\eps$,
 boundary regularity for Robin boundary value problems (see \cite{LADYZHENSKAYA} Theorem 2.1 Chap. 10) implies that $w^\eps$ is bounded in $C^{1,\gamma} (\overline {\Omega^\eps})$.
 Therefore
$w^\eps$ (resp. $\na w^\eps$) converges locally uniformly to $w_0$  (resp. $\na w_0$)  unique bounded solution of 
$$ 
\begin{cases}
w_0 -\Delta w_0 =1 & \mbox{ in } \{x\cdot n(x_0) <0\}\\
\alpha(x_0) w_0 + \beta(x_0) \nabla w_0 \cdot n(x_0) = 0& \mbox{ on } \{x\cdot n (x_0) =0\}.
\end{cases}
$$
This solution is of the form  $w_0(x) = \vphi_1(x\cdot n(x_0))$ where $\vphi_1:[-\infty,0]\to[0,1]$ solves
$ \vphi_1-\vphi_1'' = 1$ and $\alpha \vphi_1(0) +\beta\vphi_1'(0) = 0$.
We easily find 
$$
\vphi_1(x) = 1-\frac{\alpha}{\alpha+\beta} e^{x}
$$
and so $\na w_0  = \vphi_1'(0) n(x_0) =- \frac{\alpha}{\alpha+\beta}  n(x_0)$.


\end{proof}

\begin{remark}\label{rem:A1}
When $\Delta$ is replaced by an anisotropic elliptic operator as in \eqref{eq:022},  we get
$$\J_\eps(E) = - \eps \int_{\pa^* E } A \na \P \cdot \nu_E(x)\, d\H^{n-1}(x) $$
and the argument above can be adapted, with
$ 
\eps A \na\P (x_0) \to  - \frac 1 {2 \sqrt{\nu_E^T A \nu_E}} A \nu_E(x_0) = - \frac 1 2 \frac{A \nu_E(x_0) }{\| \nu_E\|_A} 
$ if $x_0\in  \pa^* E\cap \Omega$ and 
$ 
\eps A \na\P (x_0) \to - \frac{\alpha}{\alpha+\beta} \frac{ A n(x_0)}{\| n\|_A}
$ 
if $x_0\in  \pa^* E\cap \pa\Omega$ and $\pa\Omega$ is differentiable at $x_0$.
It follows that
$$ 
\lim_{\eps\to 0}
\J_\eps(E) = 
 \frac 1 2  \int_{\pa^* E } \| \nu_E\|_A   \, d\H^{n-1}(x)  +
 \int_{\pa\Omega } \frac{\alpha}{\alpha+\beta} \| n\|_A  \chi_E(x) \, d\H^{n-1}(x).  
$$ 
\end{remark}

We now turn to the proof of Theorem \ref{thm:4}, which follows from the following proposition:
\begin{proposition}\label{prop:Gamma}
Under the assumptions of Theorem \ref{thm:4}, the following holds:
\item[(i)] For any family $\{ E_\eps\}_{\eps>0}$ that converges to $E$ in $L^1(\Omega)$,
 $$\liminf_{\eps\to 0} \J_\eps(E_\eps) \geq  \F_{\frac{\alpha}{\alpha+\beta} } (E).$$

\item[(ii)] Given a set  $E$ such that $P(E,\Omega)<\infty$, there exists a sequence $\{ E_\eps\}_{\eps>0}$ that converges to $E$ in $L^1(\Omega)$ such that
 $$\limsup_{\eps\to 0} \J_\eps(E_\eps) \leq  \F_{\frac{\alpha}{\alpha+\beta} } (E).$$
\end{proposition}

\begin{proof}[Proof of Proposition \ref{prop:Gamma} - Part I: the $\liminf$ condition]
We use the formula \eqref{eq:Jepsrb} for $\J_\eps$.

\noindent{\bf Neumann boundary conditions.} We start  with the case of Neumann boundary conditions ($\alpha=0$) which is similar to the case $\Omega=\R^n$. In that case  \eqref{eq:Jepsrb} gives 
\begin{equation}\label{eq:Jneu}
\J_\eps(E)  = \eps^{-1}\int_\Omega |\chi_E-  \P |^2\, dx + \eps^{-1} \int_\Omega |\na \P| ^2\, dx
\end{equation}
and introducing the functions (both defined for $t\in [0,1]$)
$$f(t)=2 \min(t,1-t), \qquad F(t) = \int_0^t f(\tau)\, d\tau = \begin{cases} t^2 & \mbox{ for } 0\leq t\leq 1/2 \\ 2t-t^2 -\frac 1 2 & \mbox{ for } 1/2 \leq t\leq 1\end{cases} $$
we find (since $|\chi_E-  \P | = \P$ or $(1-\P)$):
\begin{align*}
\int_\Omega |\na F(\P)| \, dx  = \int_\Omega f(\P) |\na \P| \, dx 
& \leq \eps \int_\Omega  |\na \P|^2 \, dx + \eps^{-1}\int_\Omega \frac 1 4 f(\P)^2  \, dx \\
& \leq \eps\int_\Omega  |\na \P|^2 \, dx + \eps^{-1}\int_\Omega \min(\P,1-\P)^2  \, dx \\
& \leq \J_\eps (E).
\end{align*}
It follows that for a sequence $\{E_\eps \}$ such that $E_\eps\to E$ in $L^1$ and $\J_\eps(E_\eps) \leq C$, the sequence $F(\Pe)$ is bounded in $BV(\Omega)$.
Since it converges pointwise (and so $L^1$ strongly) to $F(\chi_E)=\frac 1 2 \chi_E$ (since $F(0)=0$ and $F(1) =1/2$), we deduce
$$
\liminf_{\eps\to 0} \J_\eps (E_\eps) \geq \int_\Omega |\na F(\chi_E)| \, dx = 
\frac 1 2 \int_\Omega |\na \chi_E| \, dx = \frac 1 2 P(E,\Omega).
$$
\medskip

\noindent{\bf Dirichlet boundary conditions.} The case of Dirichlet boundary conditions $\beta=0$ requires some adjustments to recover the whole perimeter $P(E)$:
We still have \eqref{eq:Jneu} in this case, but we can extend the function $\P$ by zero outside $\Omega$. Denoting by $\overline{\P}$ this extension, we find 
$$\J_\eps (E)\geq \int_\Omega |\na F(\P)| \, dx = \int_{\R^n} |\na F(\overline{\P})| \, dx $$
and so (as above)
$$
\liminf_{\eps\to 0} \J_\eps (E_\eps) \geq \int_{\R^n} |\na F(\chi_E)| \, dx = 
\frac 1 2 \int_{\R^n} |\na \chi_E| \, dx = \frac 1 2 P(E).
$$
\medskip

\noindent{\bf General Robin boundary conditions.} We can now assume that $\beta\neq 0$.
In that case, \eqref{eq:Jepsrb} and the boundary conditions give
$$
\J_\eps(E)=
\eps^{-1}\int_\Omega |\chi_E-  \P |^2\, dx + \eps^{-1} \int_\Omega |\na \P| ^2\, dx +\int_{\pa \Omega} \frac{\alpha}{\beta}|\P|^2   \, d\H^{n-1}(x).
$$
We combine the two cases above. Indeed, since $\frac{\alpha}{\beta}\geq0$, we can write, as in the Neumann case:
\begin{equation}\label{eq:11}
 \J_\eps (E_\eps) \geq \int_\Omega |\na F(\Pe)| \, dx.
\end{equation}
On the other hand, if we use the extension of $\P$ by $0$ (as in the Dirichlet case),
we can use the fact that
$$  \int_{\R^n} |\na F(\overline{\Pe})| \, dx = \int_\Omega |\na F(\Pe)| \, dx + \int_{\pa\Omega } F(\Pe) d\H^{n-1}
$$
to find
\begin{equation}\label{eq:22}
  \J_\eps (E_\eps)  \geq  \int_{\R^n} |\na F(\overline{\Pe})| \, dx +  \int_{\pa\Omega }G(\Pe) d\H^{n-1}
  \end{equation}
where the function 
$$G(t) = \frac \alpha \beta t^2 - F(t)= \begin{cases} (\frac \alpha \beta -1 )t^2 & \mbox{ for } 0\leq t\leq 1/2 \\ (\frac \alpha \beta +1) t^2-2t+ \frac 1 2 & \mbox{ for } 1/2 \leq t\leq 1\end{cases} $$  
satisfies
$G(t) \geq \min \{ 0, \frac \alpha{\alpha+\beta}-\frac1  2\}$ for all $t\in[0,1]$.
We now combine \eqref{eq:11} and \eqref{eq:22}: Given a smooth function $\vphi:\R^n\to [0,1]$, we write
\begin{align*}
 \J_\eps (E_\eps) 
 & \geq \int_\Omega |\na F(\Pe)| (1-\vphi)\, dx +  \int_{\R^n} |\na F(\overline{\Pe})| \vphi \, dx +  \int_{\pa\Omega }G(\Pe) \vphi d\H^{n-1}\\
 & \geq \int_\Omega |\na F(\Pe)| (1-\vphi)\, dx +  \int_{\R^n} |\na F(\overline{\Pe})| \vphi \, dx +  \int_{\pa\Omega } \min \left\{0, \frac \alpha{\alpha+\beta}-\frac1  2\right\} \vphi d\H^{n-1}
 \end{align*}
 which implies
\begin{align*}
\liminf_{\eps\to 0} \J_\eps (E_\eps) 
&\geq 
 \int_\Omega |\na F(\chi_E)| (1-\vphi)\, dx +  \int_{\R^n} |\na F(\chi_E)| \vphi \, dx +  \int_{\pa\Omega }  \min \left\{0, \frac \alpha{\alpha+\beta}-\frac1  2\right\} \vphi d\H^{n-1}\\
&\geq 
 \int_\Omega\frac 1 2 |\na \chi_E| \, dx +  \int_{\pa\Omega }\frac 1 2 \chi_E \vphi \, dx +  \int_{\pa\Omega }  \min \left\{0, \frac \alpha{\alpha+\beta}-\frac1  2\right\}\vphi d\H^{n-1}
 \end{align*}
and it only remains to take a sequence of $\vphi_n $ which converges to  $\chi_E$ in $L^1(\pa\Omega)$ to get the result.
\end{proof}

\begin{remark}\label{rem:A2}
The proof is easily adapted to divergence form elliptic operators as in \eqref{eq:022}. In the case of Dirichlet boundary conditions, for example, we have
\begin{align*}
\J_\eps(E) 
&  = \eps^{-1}\int_\Omega |\chi_E-  \P |^2\, dx + \eps^{-1} \int_\Omega (\na \P)^T  A  \na \P \, dx  \\
& = \eps^{-1}\int_\Omega |\chi_E-  \P |^2\, dx + \eps^{-1} \int_\Omega \| \na \P\|_A^2 \, dx\\
&\geq  \int_\Omega f(\P) \| F(\P)\|_A \, dx = \int_\Omega \|\na F(\P)\|_A \, dx =\int_{\R^n} \|\na F(\overline{\P})\|_A \, dx
\end{align*}
and so
$$ \liminf_{\eps\to 0} \J_\eps (E_\eps)\geq \frac 1 2 \int_{\R^n} \|\na \chi_E\|_A \, dx
=\frac 1 2 P^A(E).$$
\end{remark}

Before turning to the second part of the proof, we state (and prove) the following useful lemma:
\begin{lemma}\label{lem:vv}
Let $v^\eps$ be the solution of 
\begin{equation}\label{eq:vv}
\begin{cases}
v-\eps^2 \Delta v = 0 & \mbox{ in } \Omega\\
v=1 & \mbox{ on } \pa\Omega
\end{cases}
\end{equation}
and denote $\Omega_\delta = \{ x\in \Omega;\, {\mathrm dist} \, (x,\pa\Omega) >\delta\}$. Then there exists a constant $C$ (independent of $\eps$ and $\delta$) such that
$$ \eps^{-1} \int_{\Omega_{\delta}} v^\eps\, dx \leq C \mathcal H^{n-1}(\partial \Omega_\delta) \frac{\eps}{\delta} .$$
\end{lemma}
\begin{proof}
We write
$$ \eps^{-1} \int_{\Omega_{\delta}} v^\eps\, dx = \eps  \int_{\Omega_{\delta}} \Delta v^\eps\, dx = - \eps  \int_{\pa \Omega_{\delta}}\na  v^\eps\cdot n \, dx .$$
Given $x_0\in \Omega$, denote $\delta = d(x_0,\pa\Omega)$. The function $\bar v^\eps(x) = v^\eps(x_0+\eps x)$ solves
$$
\bar v^\eps - \Delta \bar v^\eps = 0 \mbox{ in } B_{\delta/\eps}(0)
$$
and satisfies $0\leq \bar v^\eps\leq 1$.
Standard elliptic estimates give
$$\eps|\na v^\eps(x_0) |=| \na \bar v^\eps(0)| \leq C \frac{\eps}{\delta}.$$
The result follows.
\end{proof}

\begin{proof}[Proof of Proposition \ref{prop:Gamma} - Part 2: The $\limsup$ condition]
Before giving the construction of $E_\eps$ in the general case, we start with the simpler cases where $\frac{\alpha}{\alpha+\beta} - \frac 1 2$ does not change sign on $\pa\Omega$.

\noindent{\bf Case 1:} When $\frac{\alpha}{\alpha+\beta} \leq  \frac 1 2 $ for all $x\in \pa\Omega$. In that case  we can simply take $E_\eps = E$ since Proposition \ref{prop:4} gives
$$ \lim_{\eps\to0} \J_\eps(E) = \F_{\frac{\alpha}{\alpha+\beta}}(E).$$

\noindent{\bf Case 2:} When $\frac{\alpha}{\alpha+\beta} >  \frac 1 2 $ for all $x\in \pa\Omega$. In that case, we have $\F_{\frac{\alpha}{\alpha+\beta}}(E) = \frac 1 2 P(E)$, and we need to approach $E$ by a sequence of sets $E_\eps$ which does not feel the effect of the boundary conditions on $\pa\Omega$:
For a given $\gamma \in(0,1)$, we define
$$ E_\eps = E\cap \Omega_{\eps^\gamma}$$
(recall that $\Omega_{\eps^\gamma} = \{ x\in\Omega\, ;\, \mathrm{dist} (x,\pa\Omega) >\eps^\gamma\}$) and we claim that 
\begin{equation}\label{eq:JEeps}
\J_\eps(E_\eps) \leq \frac 1 2 P(E_\eps) + o(1),
\end{equation}
which implies (see Giusti \cite{giusti}):
\begin{align}
\limsup_{\eps\to 0} \J_\eps(E_\eps) \leq 
\limsup_{\eps\to 0} \frac 1 2 P(E_\eps)& = \limsup_{\eps\to 0}\left[ \frac 1 2 P(E,\Omega_{\eps^\gamma}) +  \frac 1 2  \int_{\pa  \Omega_{\eps^\gamma}} \chi_E d \H^{n-1}  \right]\nonumber \\
& =\frac 1 2 P(E,\Omega) +  \frac 1 2 \int_{\pa  \Omega} \chi_E d \H^{n-1}   =\frac 1 2 P(E).\label{eq:hjgyiuy}
\end{align}
\medskip
To prove \eqref{eq:JEeps}, we write
$  \phi^\eps_{E_\eps} = v_1^\eps(x) +v_2^\eps(x)$ where $v_1^\eps = K_\eps \star \chi_{E_\eps}$ (recall that $K_\eps$ is the fundamental solution of our equation, solving $K_\eps -\eps^2 \Delta K_\eps = \delta$ in $\R^n$).
We then write
\begin{align}
\J_\eps(E_\eps)
 & = \eps^{-1} \int_\Omega  \chi_{E_\eps} (1- \phi^\eps_{E_\eps})\, dx \nonumber \\
 & = \eps^{-1} \int_\Omega  \chi_{E_\eps} (1- v_1^\eps)\, dx - \eps^{-1} \int_\Omega  \chi_{E_\eps} v_2^\eps\, dx\nonumber  \\
&= - \eps \int_{\pa^* {E_\eps} } \na v_1^\eps \cdot \nu_{E_\eps} (x)\, d\H^{n-1}(x) - \eps^{-1} \int_\Omega  \chi_{E_\eps} v_2^\eps\, dx . \label{eq:Jev1}
\end{align}
To prove that the second term goes to zero as $\eps\to0$, we note that $|v_2^\eps|\leq v^\eps $ solution of 
\eqref{eq:vv}.
Indeed, since $v_1^\eps$ solves $v_1^\eps -\eps^2 \Delta v^\eps_1 = \chi_{E_\eps}$ in $\R^n$, the maximum principle implies that $0\leq v_1^\eps\leq 1$, and so $|v_2^\eps(x)|\leq |\phi^\eps_{E_\eps} (x)-v_1(x)|\leq 1$ in $\Omega$. Furthermore the definition of $v_1^\eps$ implies that  $v_2^\eps$ solves $v_2^\eps-\eps^2 \Delta v_2^\eps = \chi_{E_\eps} -\chi_{E_\eps} =0$.
Lemma \ref{lem:vv} thus implies
$$
\left|\eps^{-1} \int_\Omega  \chi_{E_\eps} v_2^\eps\, dx\right| 
\leq 
\left|\eps^{-1} \int_{\Omega_{\eps^\gamma}} v^\eps\, dx\right| \leq C\eps^{1-\gamma}\H^{n-1}(\partial \Omega_{\eps^{\gamma}})\leq C\eps^{1-\gamma}.$$
It remains to show that 
\begin{equation}\label{eq:grad}
\sup_{\R^n}  |\eps  \na v_1^\eps (x)|\leq \frac1 2 
\end{equation} 
so that \eqref{eq:Jev1} implies
$$ \J_\eps(E_\eps) \leq \frac 1 2  \H^{n-1} (\pa^* E_\eps) + C\eps^{1-\gamma}$$
which gives \eqref{eq:JEeps}.

After a rescaling, inequality \eqref{eq:grad} follows from the following lemma:
\begin{lemma}\label{lem:half}
Given a function $f(x)$ such that $0\leq f\leq 1$,   the function  $w= K\star f$ satisfies
$$
  | \na K\star f (x) |\leq \frac1 2 \quad \mbox{ for all } x\in \R^n
$$
\end{lemma}
\begin{proof}
Without loss of generality, we can fix $x=0$ and assume that $\na w(0) = \pa_{x_n} w(0) e_n$.
We then have $  \pa_{x_n} w(0) = - \int_{\R^n} \pa_{x_n} K (y) f(y) \, dy$
and we note that $ \pa_{x_n} K\geq 0$ for $x_n<0$ and $ \pa_{x_n} K\leq 0$ for $x_n<0$. Since $0\leq f\leq 1$, we deduce
$$  \int_{\{y_n>0\} } \pa_{x_n} K (y)  \, dy \leq -\pa_{x_n} w(0) \leq  \int_{\{y_n<0\} } \pa_{x_n} K (y)  \, dy.$$
The function $\omega (x) = \int_{\{y_n<0\} }  K (x-y)  \, dy$ is the unique bounded solution of $\omega - \Delta \omega =  \chi_{\{x_n<0\}}$ which is given by $\omega(x) = \vphi(x_n) $ with
$$
\vphi(x) = \begin{cases}
1-\frac 1 2 e^{x} & \mbox{ for } x<0\\
\frac 1 2 e^{-x} & \mbox{ for } x>0\
\end{cases}$$
and satisfies 
$$ -\pa_{x_n} \omega(0) = 1/2=  \int_{\{y_n<0\} } \pa_{x_n} K (y)  \, dy .$$
\end{proof}

\medskip

\noindent{\bf General case:} Finally, we consider the general case, when $ \frac{\alpha}{\alpha+\beta} -\frac1 2 $ can change sign.
We then define $\Gamma_1=\left\{x\in \partial \Omega, \frac{\alpha}{\alpha+\beta}\leq \frac{1}{2}\right\}$ and $\Gamma_2=\left\{x\in \partial \Omega, \frac{\alpha}{\alpha+\beta}> \frac{1}{2}\right\}$.
We recall (see \eqref{eq:n2}) that $\H^{n-2}(\partial \Gamma_2)<\infty$.
For a given $\gamma\in(0,1)$, we define
$$ \Lambda_\eps =  \{x\in \Omega, \mathrm{dist}(x,\Gamma_2)>\eps^{\gamma}\}$$
and set
\[
E_{\eps}=E \cap  \Lambda_\eps.
\]
We claim that
\begin{equation}\label{eq:limsuprobin}
\J_{\eps}(E_{\eps})\leq \frac{1}{2}P(E_{\eps})+\int_{\pa \Omega}{\min\left\{0,\frac{\alpha}{\alpha+\beta}-\frac{1}{2}\right\}\chi_{E_\eps}} \, d\H^{n-1} +o(1) \qquad\mbox{ as $\eps\to 0$}.
\end{equation}
This inequality implies the result because we can proceed as in \eqref{eq:hjgyiuy} to show that 
$$ \limsup_{\eps\to 0} \frac{1}{2}P(E_{\eps}) = \frac 1 2 P(E,\Omega) +  \frac 1 2 \int_{\pa  \Omega} \chi_E d \H^{n-1}   $$
and 
$$
\left|\int_{\pa \Omega}\min\left\{0,\frac{\alpha}{\alpha+\beta}-\frac{1}{2}\right\}[\chi_{E_\eps} -\chi_E]\, d\H^{n-1}
\right| 
\leq \frac 1 2 \int_{\Gamma_1} |\chi_{E_\eps} -\chi_E|\, d\H^{n-1}
\leq \frac 1 2 \H^{n-1}(\{x\in\Gamma_1,dist(x,\Gamma_2)\leq\eps^{\gamma}\}).
$$
Using \eqref{eq:n2}, we have $\H^{n-1}(\{x\in\Gamma_1,dist(x,\Gamma_2)\leq \eps^{\gamma}\})=o(1)$, and thus
$$\limsup_{\eps\to 0} \J_{\eps}(E_{\eps})\leq  \frac 1 2 P(E,\Omega) +  \frac 1 2 \int_{\pa  \Omega} \chi_E d \H^{n-1} 
+\int_{\pa \Omega}{\min\left\{0,\frac{\alpha}{\alpha+\beta}-\frac{1}{2}\right\}\chi_{E}} \, d\H^{n-1} 
=\F_{\frac{\alpha}{\alpha+\beta}}(E).
$$
\medskip

To prove \eqref{eq:limsuprobin}, we proceed as above, writing $  \phi^\eps_{E_\eps} = v_1^\eps(x) +v_2^\eps(x)$ where $v_1 = K_\eps \star \chi_{E_\eps}$ and
\begin{align}
\J_\eps(E_\eps)
&= - \eps \int_{\pa^* {E_\eps} } \na v_1^\eps \cdot \nu_{E_\eps} (x)\, d\H^{n-1}(x) - \eps \int_{\pa^* {E_\eps} } \na v_2^\eps \cdot \nu_{E_\eps} (x)\, d\H^{n-1}(x) . \label{eq:Jev13}
\end{align}
By \eqref{eq:grad}, we have 
\[
- \eps \int_{\pa^* {E_\eps} } \na v_1^\eps \cdot \nu_{E_\eps} (x)\, d\H^{n-1}(x)\leq \frac 1 2  \H^{n-1} (\pa^* E_\eps)=\frac{1}{2}P(E_{\eps}) 
\]
which give the first term in \eqref{eq:limsuprobin}.

Since $|v_2^{\eps}|\leq v^{\eps}\leq 1$ in $\Omega$ and $v_2^{\eps}$ solves $v_2^{\eps}-\eps^2\Delta v_2^{\eps}=0$ in $\Omega$, we get (using a similar argument in Lemma \ref{lem:vv}):
\[
\eps|\nabla v_2^{\eps}(x)|\leq C\eps^{1-\gamma} \quad \mbox{ for $x\in \Omega_{\frac 1 2\eps^{\gamma}}$}.
\]
Thus the second term in \eqref{eq:Jev13} gives
\begin{align}
- \eps \int_{\pa^* {E_\eps} } \na v_2^\eps \cdot \nu_{E_\eps} (x)\, d\H^{n-1}(x)
& =- \eps \int_{\pa^* {E_\eps \setminus \Omega_{\frac 1 2\eps^{\gamma}}}} \na v_2^\eps \cdot \nu_{E_\eps} (x)\, d\H^{n-1}(x)+O(\eps^{1-\gamma})\nonumber \\
& =- \eps \int_{\pa^* {E_\eps \cap \pa \Omega}} \na v_2^\eps \cdot \nu_{E_\eps} (x)\, d\H^{n-1}(x) \nonumber \\
& \quad 
- \eps \int_{\pa^* {E_\eps \cap (\Omega \setminus \Omega_{\frac 1 2\eps^{\gamma}}) }} \na v_2^\eps \cdot \nu_{E_\eps} (x)\, d\H^{n-1}(x)+O(\eps^{1-\gamma}). \label{eq:ghteu}
\end{align}
The second term goes to zero since $\eps|\nabla v_2^{\eps}|\leq C $
and 
$$ \H^{n-1}(\pa^* {E_\eps \cap (\Omega \setminus \Omega_{\frac 1 2\eps^{\gamma}}) }
\leq P(E, \Omega \setminus \Omega_{\frac 1 2\eps^{\gamma}}) + \H^{n-1}(\pa \Lambda_\eps\cap (\Omega \setminus \Omega_{\frac 1 2\eps^{\gamma}})) \to 0.$$
Here we use \eqref{eq:n2} to show $\H^{n-1}(\pa \Lambda_\eps\cap (\Omega \setminus \Omega_{\frac 1 2\eps^{\gamma}}))\to 0$ as $\eps \to 0$.
Finally, to pass to the limit in the first term in \eqref{eq:ghteu}, we note that for $x_0\in \partial^*E_{\eps}\cap \partial \Omega \subset \Gamma_1\cap \partial^* E$, we have (using Lemma \ref{lem:3} and the definition of $v_1^{\eps}$):
\[
-\eps \nabla v_2^{\eps}(x_0)=-\eps \nabla \phi_{E_{\eps}}^{\eps}(x_0)+\eps \nabla v_1^{\eps}(x_0)\to \left(\frac{\alpha}{\alpha+\beta}-\frac 1 2\right)n(x_0).
\]
We deduce:
\begin{align*}
{\J_{\eps}(E_{\eps})}
& \leq \frac{1}{2}P(E_{\eps})+\int_{\partial^*E_{\eps}\cap \partial \Omega}\left(\frac{\alpha}{\alpha+\beta}-\frac 1 2\right) d\H^{n-1}(x)+o(1) \\
& =\frac{1}{2}P(E_{\eps})+ \int_{\Gamma_1}\left( \frac{\alpha}{\alpha+\beta}-\frac{1}{2} \right) \chi_{E_\eps} d\H^{n-1}(x) +o(1)\\
& =\frac{1}{2}P(E_{\eps}) +\int_{\partial \Omega} \min\left\{ 0,\frac{\alpha}{\alpha+\beta}-\frac{1}{2}\right\}\chi_{E_\eps}  d\H^{n-1}(x) +o(1)
\end{align*}
which is \eqref{eq:limsuprobin}.


\end{proof}


\section{The case $s\in[1,2)$ with Dirichlet or Neumann conditions}

We divide the proof in two parts depending on whether $s\in(1,2)$ or $s=1$ (the scaling of $\J_\eps^s$ makes it clear that the critical case $s=1$ is different).

\begin{proof}[Proof of Proposition \ref{prop:12} - case $s\in(1,2)$]
We write the fractional Laplacian as a divergence operator:
$$
(-\Delta)^{s/2} u = - \div D^{s-1}[u], \qquad D^{s-1}[u] = \frac{c_{n,s}}{s}\PV\int_{\R^n} u(y)  \frac{y-x}{|y-x|^{n+s}}\, dy
$$
where $D^{s-1}[u]$ is a fractional gradient of order $s-1\in(0,1)$.
\medskip

The proof is then similar to the case $s=2$: We first write
\begin{align*}
\J^s_\eps(E)& = -\eps^{s-1} \int_E \div D^{s-1}\P \, dx  \\
& = - \eps^{s-1}  \int_{\pa^* E  } D^{s-1}[ \P] \cdot \nu_E(x)\, d\H^{n-1}(x) 
\end{align*}
and we identify the limit of $\eps^{s-1} D^{s-1}[ \phi_E^\eps] \cdot \nu_E(x)$ by a blow-up argument:
Given $x_0\in \pa^* E$ we define $ w^\eps (x) = \P(x_0+\eps x)$, which solves
$$ 
\begin{cases}
 w^\eps +(-\Delta)^{s/2}  w^\eps = \chi_{ E}(x_0+\eps x) & \mbox{ in } \Omega_\eps\\
w^\eps  =0  & \mbox{ in } \C \Omega_\eps.
\end{cases}
$$
Since $ w^ \eps$, $(-\Delta)^{s/2} w^\eps$ are bounded in $L^\infty(\Omega_\eps)$ and $s>1$, 
the regularity theory for fractional elliptic equations (for instance \cite{ros-oton-Serra})  implies that $w^\eps$ is bounded in $C^{s/2}_{loc} (\R^n)$. Since $s-1<s/2$ it follows that $D^{s-1}[ w^\eps]$ is well defined and continuous in $\R^n$. 
We can then conclude as we did in the case $s=2$ thanks to the following lemma:
\begin{lemma}\label{lem:dfg}
There exist a constant $C$ such that $ |\eps^{s-1} D^{s-1} [\P] (x_0)|= |D^{s-1}[ w^\eps] (0)|\leq C$ for all $x_0 \in \pa^*E$.
Furthermore, there exists some constant $\sigma^1_{n,s}$ and $ \sigma^2_{n,s}$ (depending only on $n$ and $s$) such that
\item If $x_0\in \pa^* E\cap \Omega$, then 
$$ 
\eps^{s-1} D^{s-1}[ \P](x_0) = D^{s-1}[ w^\eps](0)\to D^{s-1}[ w_0] (0) =   \sigma^1_{n,s} \nu_E(x_0).
$$
\item If $x_0\in \pa^* E\cap \pa\Omega$, then we have 
$$ 
\eps^{s-1} D^{s-1}[ \P](x_0)=D^{s-1}[ w^\eps](0) \to D^{s-1}[w_0] (0) = \sigma^2_{n,s} n(x_0) \qquad \mbox{ if } \beta\equiv 0\quad  \mbox{(Dirichlet)}
$$
and 
$$ 
\eps^{s-1} D^{s-1}[ \P](x_0)=D^{s-1}[ w^\eps](0) \to  0\qquad   \mbox{ if } \alpha\equiv 0\quad  \mbox{(Neumann)}.
$$

\end{lemma}
\end{proof}

\begin{proof}[Proof of Lemma \ref{lem:dfg}]
The proof uses the same idea as Lemma \ref{lem:3}.

If $x_0\in \Omega$ then $\chi_{ E}(x_0+\eps x) \to \chi_{\{x\cdot \nu_E(x_0)<0\}} $ in $L^1_{loc}$  and 
for all $R>0$, $B_R(0)\subset \Omega_\eps$ for $\eps$ small enough 
Since  $w^\eps$ and $(-\Delta)^{s/2}  w^\eps$ are bounded in $L^\infty(B_R)$, it follows that $w^\eps$ is bounded in $C^\beta(B_{R/2})$  for all $\beta<s$.
In  particular, $w^\eps$ and $D^{s-1}[ w^\eps]$
converge locally uniformly  to $w_0$ and $D^{s-1}[ w_0]$ unique bounded solution of 
$$ 
w_0 +(-\Delta)^{s/2} w_0 = \chi_{\{x\cdot \nu_E(x_0)<0\}} \mbox{ in } \R^n. 
$$
This solution is of the form $w_0(x) = \vphi(x\cdot \nu_E(x_0))$
and so $D^{s-1}[ w_0]  = D^{s-1}[ \vphi](0) \nu_E(x_0)$.
The result follows with $\sigma^1_{n,s}:=D^{s-1}[ \vphi](0)$.

\medskip

If $x_0\in\pa \Omega$, and $\beta\equiv 0$, then $\nu_E(x_0)=n(x_0)$ and $\chi_{E}(x_0+\eps x) \to \chi_{\{x\cdot n(x_0)<0\}}        $ in $L^1_{loc}$
and regularity theory for fractional elliptic equations (\cite{ros-oton-Serra}) implies that 
 $w^\eps$ is bounded in $C^{s/2}_{loc} (\R^n)$.
Since $s-1<s/2$ it follows that $w^\eps$  (resp. $D^{s-1}[ w^\eps]$)
converge locally uniformly  to $w_0$ (resp. $D^{s-1}[ w_0]$)  unique bounded solution of 
$$ 
\begin{cases}
w_0 + (-\Delta )^{s/2} w_0 =1 & \mbox{ in } \{x\cdot n(x_0) <0\}\\
  w_0  = 0& \mbox{ in }  \{x\cdot n(x_0) >0\}
\end{cases}
$$
This solution is of the form  $w_0(x) = \vphi_1(x\cdot n(x_0))$ for some $\vphi_1:\R\to[0,1]$
and so 
$D^{s-1}[ w_0]  = D^{s-1}[ \vphi_1](0) n(x_0)$. The result follows with $\sigma^2_{n,s}:=D^{s-1}[ \vphi_1](0)$.

When $\alpha\equiv0$ (Neumann boundary conditions), we note that the solution of 
$$ 
\begin{cases}
w_0 + (-\Delta )^{s/2} w_0 =1 & \mbox{ in } \{x\cdot n(x_0) <0\}\\
  \widetilde {\mathcal N }(w_0 ) = 0& \mbox{ in }  \{x\cdot n(x_0) >0\}
\end{cases}
$$
is $w_0\equiv1$ which satisfies $D^{s-1}[ w_0] =0$.


 \end{proof}

\begin{proof}[Proof of Proposition \ref{prop:12} - case $s=1$]
When $s=1$, we write the fractional Laplacian as a divergence operator:
$$
(-\Delta)^{1/2} u = - \div H [u], \qquad H[u] = c_{n,1}\PV\int_{\R^n} u(y) \frac{x-y}{|x-y|^{n+1}}\, dy
$$
where $H[u]$ is an operator of order $0$ (in dimension $1$, $H$ is the Hilbert transform).
We then have
\begin{align*}
\J_\eps(E) & =  (\eps|\ln \eps|) ^{-1}\int_\Omega \chi_E (1-\P)\, dx\\
& = - (|\ln \eps|) ^{-1}\int_E (-\Delta)^{1/2} \P\, dx \\
& = (|\ln \eps|) ^{-1}\int_{\pa^* E} H[ \P] \cdot \nu_E\, d\H^{n-1}(x) 
\end{align*}
and we proceed as above:
Given $x_0\in \pa^* E$, we define $ w^\eps (x) = \P(x_0+\eps x)$, which solves
$$ 
\begin{cases}
 w^\eps +(-\Delta)^{1/2}  w^\eps = \chi_{ E}(x_0+\eps x) & \mbox{ in } \Omega_\eps\\
w^\eps  =0  & \mbox{ in } \C \Omega_\eps
\end{cases}
$$
and we conclude once again thanks to the following lemma:

\begin{lemma}\label{lem:ghj}
There exist a constant $C$ such that $ | |\ln \eps| ^{-1}  H [\P] (x_0)|= ||\ln \eps| ^{-1}  H[ w^\eps] (0)|\leq C$ for all $x_0 \in \pa^*E$.
Furthermore,
\item If $x_0\in \pa^* E\cap \Omega$, then 
$$ 
|\ln \eps| ^{-1}  H [\P] (x_0) = |\ln \eps| ^{-1}  H[\bar w^\eps] (0) \to   \sigma^1_{n,1} \nu_E(x_0).
$$
\item If $x_0\in \pa^* E\cap \pa\Omega$, then we have
$$ 
|\ln \eps| ^{-1}  H [\P] (x_0)= |\ln \eps| ^{-1} H [\bar w^\eps] (0) \to \sigma^1_{n,1} n(x_0) \qquad \mbox{ if } \beta\equiv 0 \quad  \mbox{(Dirichlet)}
$$
and 
$$ 
|\ln \eps| ^{-1}  H [\P] (x_0)= |\ln \eps| ^{-1} H [\bar w^\eps] (0) \to 0 \qquad \mbox{ if } \alpha\equiv 0 \quad  \mbox{(Neumann)}.
$$
\end{lemma}
\end{proof}

\begin{proof}[Proof of Lemma \ref{lem:ghj}]
Since $\P = 0 $ in $\C\Omega$ and $0\leq \P\leq 1$, we have  (assuming $\Omega \subset B_M$)
$$
 |\ln \eps| ^{-1}  \left| \int_{|y-x_0|\geq 1 } \P(y) \frac{x_0-y}{|x_0-y|^{n+1}}\, dy \right|
\leq |\ln \eps| ^{-1}  \int_{1\leq |y-x_0|\leq M  } \frac{1}{|x_0-y|^{n}}\, dy 
\to 0$$
so we only need to bound (and identify the limit of) the following integral:
$$
|\ln \eps| ^{-1}  \PV\int_{|y-x_0|\leq1 } \P(y) \frac{x_0-y}{|x_0-y|^{n+1}}\, dy 
=- |\ln \eps| ^{-1}  \PV\int_{|y|\leq \eps^{-1} } w^\eps(y) \frac{y}{|y|^{n+1}}\, dy .
$$

If $x_0\in \Omega$ then $B_R(0)\subset \Omega_\eps$ for $\eps$ small enough. 
Since  $w^\eps$ and $(-\Delta)^{1/2}  w^\eps$ are bounded in $L^\infty(B_R)$, it follows that $w^\eps$ is bounded in $C^\beta(B_{R/2})$  for some $\beta>0$.
In  particular, 
$$\PV\int_{|y|\leq \eps^{-1} } w^\eps(y) \frac{y}{|y|^{n+1}}\, dy  \leq 
\int_{|y|\leq 1 } \frac{|y|^{\beta} }{|y|^{n}}\, dy  
+\int_{1\leq |y|\leq \eps^{-1} }  \frac{1}{|y|^{n}}\, dy   \leq C |\ln\eps|^{-1}
$$
and the bound on  $|\ln \eps| ^{-1}  H [\P] (x_0)|$ follows.
\medskip

Next, we note that 
$w^\eps$ converge locally uniformly  to $w_0$, unique bounded solution of 
$$ 
w_0 +(-\Delta)^{1/2} w_0 = \chi_{\{x\cdot \nu_E(x_0)<0\}} \mbox{ in } \R^n .
$$
This solution is of the form $w_0(x) = \vphi(x\cdot \nu_E(x_0))$ for some function $\vphi$ satisfying in particular
$\vphi(-\infty ) =1$ and $\vphi(\infty) =0$.
Assuming (without loss of generality) that $x\cdot \nu_E(x_0)=x_1$, we get
\begin{align*}
- |\ln \eps| ^{-1}  \PV\int_{|y|\leq \eps^{-1} } \vphi(y_1) \frac{y}{|y|^{n+1}}\, dy 
& = - |\ln \eps| ^{-1}  \PV\int_{|y|\leq \eps^{-1} } \vphi(y_1) \frac{y_1}{|y|^{n+1}}\, dy \bold{e}_1 \\
& = - |\ln \eps| ^{-1}  \int_{0}^{\eps^{-1} } \frac{1}{r^{n+1}} \int_{\pa B_r} \vphi(y_1) y_1 \, d\H^{n-1}(y)\, dr \bold{e}_1 
\end{align*}
and L'Hospital Rule gives
\begin{align*}
- |\ln \eps| ^{-1}  \PV\int_{|y|\leq \eps^{-1} } \vphi(y_1) \frac{y}{|y|^{n+1}}\, dy 
& \to - \lim_{\eps\to 0} \eps^{n} 
\int_{\pa B_{\eps^{-1}}} \vphi(y_1) y_1 \, d\H^{n-1}(y) \bold{e}_1\\
& \to - \lim_{\eps\to 0}  
\int_{\pa B_{1}} \vphi(\eps y_1) y_1 \, d\H^{n-1}(y) \bold{e}_1\\
& \to 
-\int_{\pa B_{1}\cap \{y_1<0\}}   y_1 \, d\H^{n-1}(y) \bold{e}_1
\end{align*}
and the result follows with
$$ \sigma^1_{n,1}:=c_{n,1}\int_{\pa B_{1}\cap \{y_1>0\}}   y_1 \, d\H^{n-1}(y)= \frac {c_{n,1}} 2 \int_{\partial B_1} |e\cdot y|d\H^{n-1}(y).$$

\medskip

If $x_0\in\pa \Omega$ and $\beta\equiv 0$ then $\nu_E(x_0)=n(x_0)$ and regularity theory for fractional elliptic equations (\cite{ros-oton-Serra}) implies that 
 $w^\eps$ is bounded in $C^{1/2}_{loc} (\R^n)$ and so $w^\eps$ 
converges locally uniformly  to $w_0$ unique bounded solution of 
$$ 
\begin{cases}
w_0 + (-\Delta )^{1/2} w_0 =1 & \mbox{ in } \{x\cdot n(x_0) <0\}\\
  w_0  = 0& \mbox{ in }  \{x\cdot n(x_0) >0\}
\end{cases}
$$
which is of the form  $w_0(x) = \vphi_1(x\cdot n(x_0))$ for some $\vphi_1:\R\to[0,1]$ satisfying in particular
$\vphi_1(-\infty ) =1$ and $\vphi_1(\infty) =0$.
Proceeding as above, we deduce
\begin{align*}
- |\ln \eps| ^{-1}  \PV\int_{|y|\leq \eps^{-1} } \vphi_1(y_1) \frac{y}{|y|^{n+1}}\, dy 
& \to  \sigma^1_{n,1} n(x_0).
\end{align*}

When $\alpha\equiv 0$, the solution of the corresponding Neumann boundary value problem is $w_0\equiv 1$ and 
$$- |\ln \eps| ^{-1}  \PV\int_{|y|\leq \eps^{-1} } \vphi_1(y_1) \frac{y}{|y|^{n+1}}\, dy 
= - |\ln \eps| ^{-1}  \PV\int_{|y|\leq \eps^{-1} } \frac{y}{|y|^{n+1}}\, dy =0$$
which gives the last limit in Lemma \ref{lem:ghj}.
\end{proof}

\section{$\Gamma$-convergence of $\G^s_\eps(E)$ when $s\in(0,1)$}
We now consider a functional in which the nonlocal perimeter has a destabilizing effect. 
More precisely, we recall the definition of $\G^s_\eps(E)$ when $s\in(0,1)$:
\[
\G^s_\eps(E)=\displaystyle P(E,\Omega)- t \eps^{-s}\int_\Omega \chi_E (1-\P)\, dx 
\]
and we are interested in the $\Gamma$-convergence of $\G^s_\eps$ as $\eps \to 0$.
\subsection{ The case $\Omega=\R^n$ - Proof of Theorem \ref{thm:gammaRn}}
Theorem \ref{thm:gammaRn} follows from the following proposition:
\begin{proposition}\label{prop:GammaconvergenceRn}
\item[(i)]  For any set  $E$ such that $\chi_E \in L^1(\R^n)$
 $$\limsup_{\eps\to 0} \G^s_\eps(E) \leq  \G_0^s (E).$$

\item[(ii)] For any family $\{ E_\eps\}_{\eps>0}$ that converges to $E$ in $L^1(\R^n)$,
 $$\liminf_{\eps\to 0} \G^s_\eps(E_\eps) \geq  \G^s_0(E).$$

\end{proposition}
\begin{proof}
Using \eqref{eq:PsP} and Young's inequality, we find that for all $\mu$, there exists $C(\mu)$ such that
$$
P_s(E) \leq CP(E)^s|E|^{1-s}\leq \mu P(E)+C(\mu)|E|.
$$
With $\mu$ small enough, we thus have
\begin{equation}\label{eq:GP}
\G^s_0(E)=P(E)-tc_{n,s} P_s(E) \geq \frac 1 2 P(E) - C |E|.
\end{equation}

We note that (i) is obviously true when $\G^s_0(E)=\infty$, so we can assume that $\G^s_0(E)<\infty$ which, using \eqref{eq:GP} implies $P(E)<\infty$ and Theorem \ref{thm:1}-(i) implies $\lim_{\eps\to 0} \G^s_\eps(E) =  \G_0^s (E)$.
\medskip

To prove (ii), we first note that
$$
\eps^{-s}\int_{\R^n}{\chi_E(1-\phi_E^{\eps})}=c_{n,s} \int_{\R^n} \P (\chi_E \psi_{\C E} - \chi_{\C E} \psi_E)\, dx\leq c_{n,s}P_s(E)
$$
and so (proceeding as above), 
$$ 
\G^s_\eps(E) \geq \frac 1 2 P(E) - C |E|.
$$
We note that (ii) is trivially true if $\liminf_{\eps\to 0}\G^s_\eps(E_\eps)=\infty$, so we can assume that
$\liminf_{\eps\to 0}\G^s_\eps(E_\eps)<C$, and we fix a subsequence 
 $\{E_{\eps_k}\}$ such that
$\liminf_{\eps\to 0}\G^s_\eps(E_\eps)=\lim_{\eps_k\to 0}{\G^s_{\eps_k}(E_{\eps_k})}$ and $\G^s_{\eps_k}(E_{\eps_k})<C$.
Taking $\mu=\frac{1}{2t}$ and we get
\[
\frac 1 2 P(E_{\eps_k})\leq \G^s_{\eps_k}(E_{\eps_k}) + C |E_{\eps_k}|<C.
\]
We can use this bound to show that (recall \eqref{eq:int1})
\[
J_k=\eps_k^{-s}\int_{\R^n}{\chi_{E_{\eps_k}}(1-\phi_{E_{\eps_k}}^{\eps_k})}=c_{n,s} \int_{\R^n} \phi_{E_{\eps_k}} ^{\eps_k}(\chi_{E_{\eps_k}} \psi_{\C {E_{\eps_k}}} - \chi_{\C {E_{\eps_k}}} \psi_{E_{\eps_k}})\, dx
\]
converges to 
\[
J_{\infty}=c_{n,s}\int_{\R^n}{\chi_E\psi_{\C E}} = c_{n,s} P_s(E).
\]
Indeed, we have
\begin{align*}
\frac{1}{c_{n,s}}|J_k-J_{\infty}|&\leq \int_{\R^n} \phi_{E_{\eps_k}} ^{\eps_k}|\chi_{E_{\eps_k}} \psi_{\C {E_{\eps_k}}} - \chi_{\C {E_{\eps_k}}} \psi_{E_{\eps_k}}-\chi_E \psi_{\C E} + \chi_{\C E} \psi_E|\, dx\\
&\quad +\left|\int_{\R^n} \phi_{E_{\eps_k}} ^{\eps_k}(\chi_E \psi_{\C E} - \chi_{\C E} \psi_E)\, dx-\int_{\R^n}{\chi_E \psi_{\C E}}\right|.
\end{align*}
We can proceed as in the proof of Theorem \ref{thm:1}-(i) to prove that the second term goes to $0$ as $\eps_k \to 0$ by. The first term is bounded by
\[
\int_{\R^n} |\chi_{E_{\eps_k}} \psi_{\C {E_{\eps_k}}} - \chi_{\C {E_{\eps_k}}} \psi_{E_{\eps_k}}-\chi_E \psi_{\C E} + \chi_{\C E} \psi_E|\, dx\leq 2P_s(E_{\eps_k}\Delta E)\leq CP(E_{\eps_k}\Delta E)^s|E_{\eps_k}\Delta E|^{1-s}\to 0
\] 
since $E_{\eps_k}\to E$ in $L^1(\R^n)$ and $P(E),P(E_{\eps_k})<\infty$. 
Finally, we have $\G^s_{\eps_k} (E_{\eps_k})=\displaystyle P(E_{\eps_k}) - t J_k$, so the result now follows from the lower semicontinuity of the perimeter.

\end{proof}

\subsection{ The case $\Omega\neq \R^n$ with Robin boundary conditions}
We need to prove:

\begin{proposition}\label{prop:Gammaconvergencerobin}
\item[(ii)] For any set  $E$ such that $\chi_E\in L^1(\R^n)$, 
 $$\limsup_{\eps\to 0} \G^s_\eps(E) \leq  \G_0^s (E).$$

\item[(ii)] For any family $\{ E_\eps\}_{\eps>0}$ that converges to $E$ in $L^1(\Omega)$,
 $$\liminf_{\eps\to 0} \G^s_\eps(E_\eps) \geq  \G^s_0(E).$$

\end{proposition}

\begin{proof}
Using \eqref{eq:PsLP}, we find
$$P_s^L(E,\Omega)
\leq C(P(E,\mathrm{Conv}(\Omega)))^s|E\cap \Omega|^{1-s}
\leq \mu P(E,\mathrm{Conv}(\Omega))+C(\mu)|E|.$$
If $\Omega$ is convex, we have $ P(E,\mathrm{Conv}(\Omega))= P(E, \Omega)$. In the general case, we assume that $P(\Omega)<\infty$ and write $P(E,\mathrm{Conv}(\Omega)) \leq P(E)= P(E,\Omega) +P(\Omega)$ (recall that $E\subset \Omega$) to get
$$P_s^L(E,\Omega)
\leq \mu P(E,\Omega)+\mu P(\Omega)+C(\mu)|E|.$$
Since $0\leq \frac{\alpha}{\alpha+\beta},\frac{\beta}{\alpha+\beta} \leq 1$, we have
$$\int_{\C\Omega}  \frac{\alpha}{\alpha+\beta} \psi_{E }(x)\, dx\leq 
\int_{\C \Omega}{\psi_Edx}\leq \int_{\C \Omega}{\psi_\Omega dx}= P_s(\Omega)<\infty,$$
and
$$\int_{\C\Omega}   \frac{\beta}{\alpha+\beta}	 \frac{ \psi_{E\cap\Omega} (x)\psi_{\C E \cap \Omega} (x)}{\psi_\Omega(x)} \, dx \leq \int_{\C \Omega}{\frac{\psi_{E\cap \Omega}\psi_{\C E\cap \Omega}}{\psi_{\Omega}}dx}\leq \int_{\C \Omega}{\int_{\C E \cap \Omega}\frac{1}{|x-y|^{n+s}dy}dx}\leq P_s(\Omega)<\infty.$$
It follows that 
$$ 
\G^s_0(E) \geq \frac 1 2 P(E,\Omega) - CP(\Omega)-C |E| - CP_s(\Omega).
$$

Since the statement (i) is true if $\G^s_0(E)=\infty$, we can assume that $\G^s_0(E)<\infty$, which now implies that $P(E,\Omega)<\infty$. The result thus follows from Theorem \ref{thm:2}-(i).
\medskip

For (ii), 
we note, using \eqref{eq:introbin} and the fact that $0\leq \P\leq 1$, that
\begin{align*}
\eps^{-s}\int_{\Omega}{\chi_E(1-\phi_E^{\eps})}&=C_{n,s} \int_{\Omega} {\P (\chi_E \psi_{\C E} - \chi_{\C E} \psi_E)\, dx} -C_{n,s}\int_{\C \Omega}{\int_{\Omega}{\frac{\beta}{\alpha+\beta}\frac{\psi_{E}(x)}{\psi_{\Omega}(x)}\frac{\P}{|x-y|^{n+s}}dy}dx}\\
&\leq 2C_{n,s}P_s^L(E,\Omega).
\end{align*}
Proceeding as  in the proof of Proposition \ref{prop:GammaconvergenceRn}, we can assume that $ \liminf_{\eps\to 0}\G^s_\eps(E_\eps)=\lim_{\eps_k\to 0}{\G^s_{\eps_k}(E_{\eps_k})}<\infty$
 and  $P(E_{\eps_k})\leq C$ and we need to show that 
\begin{align*}
J_k:=\frac{1}{C_{n,s}}\eps_k^{-s}\int_{\Omega}{\chi_{E_{\eps_k}}(1-\phi_{E_{\eps_k}}^{\eps_k})}&= \int_{\R^n} \phi_{E_{\eps_k}} ^{\eps_k}(\chi_{E_{\eps_k}} \psi_{\C {E_{\eps_k}}} - \chi_{\C {E_{\eps_k}}} \psi_{E_{\eps_k}})\, dx\\
&\quad -\int_{\C \Omega}{\int_{\Omega}{\frac{\beta(x)}{\alpha(x)+\beta(x)}\frac{\psi_{E_{\eps_k}}(x)}{\psi_{\Omega}(x)}\frac{\phi_{E_{\eps_k}}^{\eps_k}}{|x-y|^{n+s}}dy}dx}
\end{align*}
converges to
\[
J_{\infty}:=\int_{\Omega}{\chi_E\psi_{\C E}}+\int_{\C \Omega}{\frac{\alpha}{\alpha+\beta}\psi_{ E}}+\int_{\C \Omega}{\frac{\beta}{\alpha+\beta}\frac{\psi_{E\cap \Omega}\psi_{\C E\cap \Omega}}{\psi_{\Omega}}}.
\]
We have
\begin{equation}\label{eq:ineqgammaconvergence}
\begin{aligned}
|J_k-J_{\infty}|&\leq \int_{\Omega} \phi_{E_{\eps_k}} ^{\eps_k}|\chi_{E_{\eps_k}} \psi_{\C {E_{\eps_k}}} - \chi_{\C {E_{\eps_k}}} \psi_{E_{\eps_k}}-\chi_E \psi_{\C E} + \chi_{\C E} \psi_E|\, dx\\
&+\left|\int_{\C \Omega}{\int_{\Omega}{\frac{\beta(x)}{\alpha(x)+\beta(x)}\frac{\phi_{E_{\eps_k}}^{\eps_k}(y)}{|x-y|^{n+s}}\frac{\psi_{E_{\eps_k}}(x)-\psi_{E}(x)}{\psi_{ \Omega}(x)}}}\right|\\
&+\left|\int_{\Omega}{\phi_{E_{\eps_k}} ^{\eps_k}(\chi_E \psi_{\C E} - \chi_{\C E} \psi_E)}-\int_{\C \Omega}{\int_{\Omega}{\frac{\beta(x)}{\alpha(x)+\beta(x)}\frac{\phi_{E_{\eps_k}}^{\eps_k}(y)}{|x-y|^{n+s}}\frac{\psi_{E}(x)}{\psi_{ \Omega}(x)}}}-J_{\infty}\right|.
\end{aligned}
\end{equation}
We can proceed as in the proof of Theorem \ref{thm:2} to show that the last term goes to $0$ as $k \to \infty$.
The first term is bounded by
\begin{align*}
&  \int_{\Omega} |\chi_{E_{\eps_k}} \psi_{\C {E_{\eps_k}}} - \chi_{\C {E_{\eps_k}}} \psi_{E_{\eps_k}}-\chi_E \psi_{\C E} + \chi_{\C E} \psi_E|\, dx\\
&\qquad \leq 2P_s^L(E_{\eps_k}\Delta E,\Omega)\leq C(P(E_{\eps_k}\Delta E,\Omega)+P(\Omega))^s|E_{\eps_k}\Delta E|^{1-s} \longrightarrow 0
\end{align*}
since $E_{\eps_k}\to E$ in $L^1(\Omega)$ and $P(E,\Omega),P(E_{\eps_k},\Omega),P(\Omega)<\infty$. The second term is bounded by
\begin{align*}
 \int_{\C \Omega}{\int_{\Omega}{\frac{1}{|x-y|^{n+s}}\frac{|\psi_{E_{\eps_k}}(x)-\psi_{E}(x)|}{\psi_{\Omega}(x)}dy}dx}&=\int_{\C \Omega}{\int_{\Omega}{\frac{|\chi_{E_{\eps_k}}(y)-\chi_E(y)|}{|x-y|^{n+s}}dy}dx}\\
 &=\int_{\Omega}{\psi_{\C \Omega}(y)|\chi_{E_{\eps_k}}(y)-\chi_E(y)|dy}.
\end{align*}
Since $\psi_{\C\Omega}(y)\in L^1(\Omega)$, $|\chi_{E_{\eps_k}}(y)-\chi_E(y)|\leq 1$ and  $|\chi_{E_{\eps_k}}(y)-\chi_E(y)|\to 0$ a.e., Lebesgue's dominated convergence theorem yields
$$\int_{\C \Omega}{\int_{\Omega}{\frac{|\chi_{E_{\eps_k}}(y)-\chi_E(y)|}{|x-y|^{n+s}}}dx} \to 0$$
as $\eps_k\to 0$.
This proves that the right hand side in \eqref{eq:ineqgammaconvergence} goes to zero, and thus $J_k\to J_\infty$.
Since 
$\G^s_{\eps_k} (E_{\eps_k})=\displaystyle P(E_{\eps_k},\Omega) - t J_k$,
we conclude the proof of (ii) by using the lower semicontinuity of the perimeter.
\end{proof}

\section{$\Gamma$-convergence of $\G_\eps(E)$ (case $s=2$)}
Theorem \ref{thm:gammaOmega2} follows from the following proposition:
\begin{proposition}\label{prop:Gammaconvergenceneuman}
Assume that $\Omega$ is a bounded set with $C^{1,\alpha}$ boundary and that $t<2$.
\item[(i)] For any set  $E$ such that $\chi_E\in L^1(\R^n)$, 
 $$\limsup_{\eps\to 0} \G_\eps(E) \leq  \G_0 (E) = \left( 1-\frac t 2\right) P(E,\Omega).$$
\item[(ii)] For any family $\{ E_\eps\}_{\eps>0}$ that converges to $E$ in $L^1(\Omega)$,
 $$\liminf_{\eps\to 0} \G_\eps(E_\eps) \geq  \G_0(E).$$

\end{proposition}

The key tool in the proof of this proposition  is the following lemma:
\begin{lemma} \label{lem:Pgrad}
Assume that $\Omega$ is a bounded set with $C^{1,\alpha}$ boundary.
For any measurable set $E \subset \Omega$, 
and for any $\alpha'\in(0,\alpha)$,
the solution $\P$ of \eqref{eq:03} satisfies
\begin{equation}\label{eq:phigrad}
|\eps \na \P(x)| \leq \frac 1 2 + C\left(\frac 1 R + \eps^{\alpha'} R^{1+\alpha'} \right) \mbox{ for all $R>0$, $\eps>0$, $x\in\Omega$} 
\end{equation}
where the constant $C$ depends on $\Omega$ but not on $E$.
\end{lemma}

\begin{proof}[Proof of Proposition \ref{prop:Gammaconvergenceneuman}]
The statement (i) follows from Proposition \ref{prop:4}.


To prove (ii), we recall that
\begin{align*}
\eps^{-1} \int_\Omega  \chi_E (1-\P)\, dx
&  =  -\eps \int_E \Delta \P \, dx  \\
& = - \eps \int_{\pa^* E \cap\Omega} \na \P \cdot \nu_E(x)\, d\H^{n-1}(x) 
\end{align*}
so that  Lemma \eqref{lem:Pgrad} implies
$$
\eps^{-1} \int_\Omega  \chi_E (1-\P)\, dx
\leq
\left( \frac 1 2 + C\left(\frac 1 R + \eps^{\alpha'} R^{1+\alpha'} \right) \right) 
P(E,\Omega)
$$
for all $R>0$ and $\eps>0$ (with a constant $C$ independent of $E$).
For any  $0< \eta < 1-\frac t 2$, we choose $R$ such that $C/R<\frac 1 2\eta/t$ and then $\eps_0$ such that $C \eps^{\alpha'} R^{1+\alpha'}<\frac 1 2\eta/t$ for all $\eps<\eps_0$, . We then have
$$ \G_\eps(E)\geq  \left( 1-\frac t 2 -\eta\right) P(E,\Omega) \qquad \forall \eps<\eps_0$$ 
for all set $E$ with $P(E,\Omega)<\infty$.
The lower semincontinuity of the perimeter implies
$$\liminf_{\eps\to 0} \G_\eps(E_\eps) \geq  \left( 1-\frac t 2 -\eta \right) P(E,\Omega).$$
Since this holds for all $\eta>0$, the result follows.
\end{proof}

\begin{proof}[Proof of Lemma \ref{lem:Pgrad}]
We fix $R>0$ throughout the proof.
First, we prove that \eqref{eq:phigrad} holds for all $x_0$ such that $B_{R\eps}(x_0)\subset \Omega$.
For that, we consider the functions $v^\eps = K_\eps * \chi_E$ and $u^\eps=\P-v^\eps$.
Lemma \ref{lem:half} implies that $|\eps \na v^\eps(x_0)  | \leq 1/2$.
The rescaled function $\bar u^\eps = u^\eps(x_0+\eps x)$ solves
$ \bar u^\eps - \Delta \bar u^\eps=0 $ in $B_R(0)$ and satisfies $\| \bar u^\eps\| \leq 2$. We deduce that
$ |\na \bar u (0)|  \leq  \frac C R$ and therefore
$$ |\eps \na \P(x_0)| \leq |\eps \na v^\eps(x_0)  | +\frac C R \leq \frac 1 2 +\frac C R.$$

Next, we take $y_0\in \pa\Omega$ and we are going to show that \eqref{eq:phigrad} holds in $B_{R\eps} (y_0)$. Together with the argument above, this implies the lemma. This part of the proof is more delicate and must make use of the Neumann condition and of the regularity of $\pa\Omega$.
Without loss of generality, we assume that $y_0=0$ and $\nu(y_0) = -e_n$. In what follows, we use the notation $x=(x',x_n)$ with $x'\in \R^{n-1}$ and $x_n\in \R$.
We define $F(x)= \chi_E (x',x_n)+\chi_E (x',-x_n)$ the even extension of $\chi_E$ and denote 
$ v^\eps_1(x) = K_\eps * F$, which satisfies $|\eps \na v^\eps_1(x_0)  | \leq 1/2$ by Lemma \ref{lem:half}.
Note that $v^\eps_1$ solves the equation $v^\eps_1 - \eps^2 \Delta v^\eps_1 = \chi_E$ in $\{x_n>0\}$ and satisfies the Neumann boundary condition on $\{ x_n=0\}$.
The idea of the proof is to use the $C^{1,\alpha}$ regularity of $\Omega$ to say that in $B_{R\eps}$, $\Omega$ is close to the half space $\{x_n>0\}$ and that the function $u^\eps = \P - v^\eps_1$ (and its gradient) is small.

The rescaled function $\bar u^\eps(x)=u^\eps(y_0+\eps x) $ solves
\begin{equation}\label{eq:bun}
\begin{cases}
\bar u^\eps - \Delta \bar u^\eps = 0 & \mbox{ in } \Omega^\eps \\
\na \bar u^\eps \cdot \nu^\eps = g^\eps(x) & \mbox{ on } \pa\Omega^\eps
\end{cases}
\end{equation}
where
$$ g^\eps(x) = -\na \bar v_1^\eps(x)\cdot \nu^\eps (x),$$
with $\bar v_1^\eps(x)=v_1^{\eps}(y_0+\eps x)$ and $\Omega^\eps = \{ x\, ;\, y_0+\eps x\in \Omega\}$.

Gradient estimates for Neumann boundary value problems (see \cite{LADYZHENSKAYA,Koenig}) give, for $\eta>0$:
\begin{align}
\sup_{\Omega^\eps \cap B_{R}} |\na \bar u^\eps| 
& \leq \frac 1 R  \| \bar u^\eps\|_{L^\infty(\Omega^\eps \cap B_{2R})}
+ C  \| g^\eps\|_{L^\infty (\pa\Omega^\eps \cap B_{2R}) }+ R^\eta [ g^\eps]_{C^\eta (\pa\Omega^\eps \cap B_{2R}) }.
\label{eq:buepsb}
\end{align}
We have
$$\| \bar u^\eps\|_{L^\infty(\Omega^\eps \cap B_{2R})} \leq \|u^\eps\|_{L^\infty(\Omega )}\leq\|\P \|_{L^\infty(\Omega) }+\|v_1^\eps\|_{L^\infty(\Omega )}\leq 2,$$
and we can thus conclude thanks to the following Lemma:
\begin{lemma}\label{lem:geps}
For all $\sigma<1$ there exists $C>0$ such that 
$$
 \| g^\eps\|_{L^\infty (\pa\Omega^\eps \cap B_{3R}) } \leq C\left[ (\eps R)^\alpha + (\eps^\alpha R^{1+\alpha})^\sigma \right]
$$
and
$$
\| g^\eps\|_{C^\eta (\pa\Omega^\eps \cap B_{2R}) }
\leq C\left[ (\eps R)^\alpha + (\eps^\alpha R^{1+\alpha})^\sigma \right]^{1-\eta/\sigma}.
$$
\end{lemma}
Lemma \ref{lem:geps} together with \eqref{eq:buepsb}   yields
$$ \sup_{\Omega^\eps \cap B_{R}} |\na \bar u^\eps| \leq C\left(\frac 1 R + \eps^{\alpha'} R^{1+\alpha'} \right)
$$
for some $\alpha'=\alpha (\sigma-\eta) <\alpha$.
Since $|\na v_1^\eps|\leq 1/2$, we deduce:
$$ 
\sup_{\Omega \cap B_{R\eps}} |\eps \na \P | \leq\frac 1 2 +  C\left(\frac 1 R + \eps^{\alpha'} R^{1+\alpha'} \right)
$$
which complete the proof \eqref{eq:phigrad}.
\end{proof}

\begin{proof}[Proof of Lemma \ref{lem:geps}]
We recall that $ g^\eps(x) = -\na \bar v_1^\eps(x)\cdot \nu^\eps (x)$ where 
$\bar v^\eps_1(x) = K * \bar F$. Calder\'on-Zygmund's estimates gives $\|  \bar v^\eps_1(x) \|_{W^{2,p}(B_{3R})} \leq C$ for all $p<\infty$ 
and so $\|\na \bar v^\eps_1(x)\|_{C^\sigma(B_{3R})} \leq C $.
We note that
$$g^\eps(x) = -\na_{x'} \bar v_1^\eps(x)\cdot \nu^\eps_{x'} (x) - \pa_{x_n} \bar v_1^\eps(x)\nu_n^\eps (x)
$$
where the $C^{1,\alpha}$ regularity of $\Omega$ implies
$$ 
|  {\nu^\eps}' (x) | \leq C (\eps R)^\alpha \quad \mbox{ in } B_{\eps R}(y_0) 
$$
while the fact that $\pa_{x_n}  v^\eps_1=0$ on $\{x_n=0\}$ yields:
$$ 
|\pa_{x_n} \bar v_1^\eps(x)|\leq C|x_n|^\sigma \leq C(\eps^\alpha R^{1+\alpha} )^\sigma.
$$
The first estimate in Lemma \ref{lem:geps} follows.
The second estimate then follows from the interpolation inequality
$$
\| g^\eps\|_{C^\eta} \leq C \| g^\eps \|_{C^\sigma}^{\eta/\sigma} \| g^\eps\| _{L^\infty} ^{1-\eta/\sigma} 
\leq C \|\na \bar v^\eps_1(x)\|_{C^\sigma(B_{3R})}^{\eta/\sigma} \| g^\eps\| _{L^\infty} ^{1-\eta/\sigma} .
$$

\end{proof}

\appendix

\section{An interpolation inequality}
We recall that 
$$
P_s^L(E,\Omega) 
:= \int_{\Omega} \int_{\Omega} \frac{ \chi_E (x) \chi_{\C E}(y)}{|x-y|^{n+s}} \, dx \, dy
= \frac 1 2 \int_{\Omega} \int_{\Omega} \frac{ |\chi_E (x)- \chi_{E}(y)|}{|x-y|^{n+s}} \, dx \, dy.
$$
We will prove that for any function $u\in W^{1,1}(\Omega)$ we have
\begin{equation}\label{eq:uu}
\int_{\Omega} \int_{\Omega} \frac{ |u(x)- u(y)|}{|x-y|^{n+s}} \, dx \, dy
\leq 
\frac{n\omega_n 2^{1-s}}{s(1-s)} \| u\|_{L^1(\Omega)}^{1-s} \left( \int_{\mathrm {Conv}(\Omega) } |\na u|\, dx\right)^s
\end{equation}
so the result follows by a density argument.\\
We split the integral in the right hand side between $\{(x,y)\in \Omega^2\, ;\, |x-y|\geq R\}$ and $\{(x,y)\in \Omega^2\, ;\, |x-y|\leq R\}$.
For the first one, we write
\begin{align*}
 \int\int_{\{(x,y)\in \Omega^2\, ;\, |x-y|\geq R\} }   \frac{ |u (x)- u(y)|}{|x-y|^{n+s}} \, dx \, dy
&  \leq  \int\int_{\{(x,y)\in \R^{2n}\, ;\, |x-y|\geq R\} }   \frac{|u (x)|+ |u(y)|}{|x-y|^{n+s}} \, dx \, dy \\
& \leq 2n\omega_n\| u\|_{L^1(\Omega)}\, dx  \frac{R^{-s}}{s}.
\end{align*}
For the second integral, we write
\begin{align*}
&  \int\int_{\{(x,y)\in \Omega^2\, ;\, |x-y|\leq R\} }   \frac{ |u (x)- u(y)|}{|x-y|^{n+s}} \, dx \, dy\\
& 
= \int\int_{\{(x,y)\in \Omega^2\, ;\, |x-y|\leq R\} }   \frac{ 1}{|x-y|^{n+s}} \int_0^1 |\na u (tx+(1-t)y) \cdot (x-y)|\, dt \, dx \, dy\\
& 
\leq \int_0^1 \int_\Omega \int_{\{z\in \Omega_t(y)\, ;\, |z-y|\leq Rt\} }   \frac{ t^{s-1}}{|z-y|^{n-1+s}}  |\na u(z) | \, dz \, dy\, dt
\end{align*}
where we did the change of variable $ z= tx+(1-t)y$, so that $\Omega_t(y) = \{ z = tx+(1-t)y\,;\, x\in \Omega\}$. 
It is readily seen that for $t\in[0,1]$ and $y\in\Omega$, we have $\Omega_t(y)  \subset \mathrm {Conv}(\Omega) $ the convex hull of $\Omega$.
We deduce:
\begin{align*}
&  \int\int_{\{(x,y)\in \Omega^2\, ;\, |x-y|\leq R\} }   \frac{ |u (x)-u(y)|}{|x-y|^{n+s}} \, dx \, dy\\
& 
\leq \int_0^1 \int_{\mathrm {Conv}(\Omega)}  \int_{\{y\in \Omega\, ;\, |z-y|\leq Rt\} }   \frac{ t^{s-1}}{|z-y|^{n-1+s}}  |\na u(z)| \, dy \, dz\, dt\\
& 
\leq \int_{\mathrm {Conv}(\Omega)}|\na u(z)|\, dz  \int_0^1  \int_{\{|z-y|\leq Rt\} }   \frac{ t^{s-1}}{|z-y|^{n-1+s}} \, dy \, dt\\
& 
\leq 
n\omega_n \frac{R^{1-s}}{1-s}
\int_{\mathrm {Conv}(\Omega)}|\na u(z)|\, dz.
\end{align*}
We deduce 
$$ 
\int_{\Omega} \int_{\Omega} \frac{ |u(x)- u(y)|}{|x-y|^{n+s}} \, dx \, dy
 \leq 2n\omega_n \| u\| _{L^1(\Omega)}  \frac{R^{-s}}{s}+ n\omega_n \frac{R^{1-s}}{1-s}
\int_{\mathrm {Conv}(\Omega)}|\na u (z)|\, dz
$$
and \eqref{eq:uu} follows by taking $R=\frac{2\| u\| _{L^1(\Omega)} }{\int_{\mathrm {Conv}(\Omega)}|\na u(z)|\, dz}$.

\bibliographystyle{siam}
\bibliography{perimeter}

\end{document}